\documentclass[11pt]{amsart}

\usepackage{eucal}
\usepackage{amssymb}
\usepackage{amsthm}
\usepackage{enumitem}
\usepackage{amsmath}
\usepackage{accents}
\usepackage[all,2cell]{xy}
\UseAllTwocells
\usepackage{setspace}

\usepackage[a4paper, margin=2.5cm]{geometry}

\usepackage[hidelinks,pdfencoding=unicode]{hyperref}

\title{Orientals as free algebras}

\author{Dimitri Ara}
\address{Dimitri Ara, Aix~Marseille~Univ,~CNRS,~I2M,~Marseille,~France}
\email{dimitri.ara@univ-amu.fr}

\author{Yves Lafont}
\address{Yves Lafont, Aix~Marseille~Univ,~CNRS,~I2M,~Marseille,~France}
\email{yves.lafont@univ-amu.fr}

\author{Fran\c{c}ois M\'etayer}
\address{Fran\c{c}ois M\'etayer, Universit\'e Paris Cit\'e, CNRS, IRIF, F-75013, Paris, France \&
Universit\'e Paris Nanterre}
\email{metayer@irif.fr}

\subjclass{18N30, 18N50, 55U10}
\keywords{augmented directed complexes, cones, expansions, polygraphs,
orientals, simplices, strict \oo-categories}

\newcommand\letenv[2]{%
\expandafter\expandafter\expandafter\let\expandafter\expandafter
\csname #1\endcsname\csname #2\endcsname
\expandafter\expandafter\expandafter\let\expandafter\expandafter
\csname end#1\endcsname\csname end#2\endcsname
}

\swapnumbers
\theoremstyle{plain}
\newtheorem{theorem}{Theorem}[subsection]
\newtheorem{proposition}[theorem]{Proposition}
\newtheorem{lemma}[theorem]{Lemma}
\newtheorem{corollary}[theorem]{Corollary}
\theoremstyle{definition}
\newtheorem{definition}[theorem]{Definition}
\newtheorem{remark}[theorem]{Remark}
\newtheorem{remarks}[theorem]{Remarks}
\newtheorem{example}[theorem]{Example}
\newtheorem{examples}[theorem]{Examples}
\let\paragraph\undefined
\newtheorem{paragraph}[theorem]{}
\letenv{paragr}{paragraph}

\newcommand\ndef\emph
\newcommand\nbd\nobreakdash
\newcommand\oo{$\omega$\nbd}
\newcommand\asterism\medskip

\renewcommand\descriptionlabel[1]{\hspace{\labelsep}\textit{#1}} 

\newcommand \myspace {\vspace{1.5ex}}

\newcommand\e\epsilon

\newcommand \set[1] {\{#1\}}
\newcommand \setbis[2] {\left{#1\; |\; #2\right}}

\renewcommand \emptyset \varnothing
\newcommand \sing 1

\newcommand\N{\mathbb{N}}
\newcommand\Z{\mathbb{Z}}

\newcommand \Sing {\mathbf 1}
\newcommand \Set {\mathbf{Set}}
\newcommand \Cat {\mathbf{Cat}}

\newcommand \sketch \Sigma
\newcommand \esketch {\sketch_e}

\newcommand \Mod[1] {\operatorname{Mod}(#1)}

\newcommand{\scat}{\Delta}
\newcommand{\sset}{\widehat{\scat}}
\newcommand {\nrv}{N} 
\newcommand\deltan[1]{\Delta_{#1}} 

\newcommand{\nglob}[1]{\mathbf{Glob}_{#1}}
\newcommand{\oglob}{\nglob{\omega}}

\newcommand{\ncat}[1]{\mathbf{Cat}_{#1}}
\newcommand \ncatpl[1] {\ncat{#1}^{+}} 
\newcommand{\ocat}{\ncat{\omega}}
\newcommand\ooCat\ocat
\newcommand{\ocatpt}{\ncat{\omega, \ast}} 
\newcommand{\npol}[1]{\mathbf{Pol}_{#1}}
\newcommand{\opol}{\npol{\omega}}

\newcommand{\doubl}[2]{\ar@<2pt>[l]^{#2}\ar@<-2pt>[l]_{#1}}
\newcommand{\doubr}[2]{\ar@<2pt>[r]^{#1}\ar@<-2pt>[r]_{#2}}
\newcommand{\Doubl}{\doubl{}{}}
\newcommand{\doubld}[2]{\ar@<2pt>[ld]^{#2}\ar@<-2pt>[ld]_{#1}}

\newcommand{\DOUBL}{\ar@<2pt>[l]^(.4){\SCE}\ar@<-2pt>[l]_(.4){\TGE}}
\newcommand{\DOUBLD}{\ar@<2pt>[ld]^(.4)*-<4pt>{\scriptstyle{\SCE}}\ar@<-2pt>[ld]_(.4)*-<4pt>{\scriptstyle{\TGE}}}

\newcommand \Right[1] {U_{#1}}
\newcommand \Left[1] {L_{#1}}
\newcommand \Free[1] {F_{#1}}
\newcommand \oFree {\Free \omega}
\newcommand \free[1] {#1^{\ast}}

\newcommand \comp[1] {\ast_{#1}}
\newcommand \unit[1]{1_{#1}}
\newcommand \Unit[2] {1^{#1}_{#2}}

\newcommand{\adjd}[2]{\ar@<1ex>[d]^{#1}_{\dashv}\ar@<-1ex>@{<-}[d]_{#2}}
\newcommand{\adjr}[2]{\ar@<1ex>[r]^{#1}_{\top}\ar@<-1ex>@{<-}[r]_{#2}}

\newcommand \cyl{\Gamma}
\newcommand \Cyl[1]{\cyl(#1)}
\newcommand \nCyl[2]{\cyl_{#1}(#2)}

\newcommand \cto{\curvearrowright}

\newcommand \Top {\pi^1}
\newcommand \Bot {\pi^2} 

\newcommand \Triv{\tau}

\newcommand \orient{\mathcal{O}}
\newcommand \orienta{\mathcal{O}_{+}}

\newcommand \norient[1]{\orient_{#1}}

\newcommand \faced[2]{\delta_{#1}^{#2}}
\newcommand \degend[2]{\sigma_{#1}^{#2}}

\newcommand \NN {\mathbb N} 

\newcommand \homset[3]{#1(#2,#3)}

\renewcommand \epsilon \varepsilon

\renewcommand \setbis[2] {\{#1 \; | \; #2\}} 

\newcommand \sce[1] {#1^{-}} 
\newcommand \tge[1] {#1^{+}} 

\renewcommand \comp[1] {\mathop{\ast_{#1}}} 

\renewcommand \Top {\overline \pi} 
\renewcommand \Bot {\underline \pi} 

\newcommand \expcat {\mathbf{Cat}_{\omega, e}}  

\newcommand \proj \Pi 

\newcommand \trunc[2] {#2^{(#1)}} 

\newcommand \expan[1]{#1^{\lhd}} 
\newcommand \fexpan[1]{#1^{{\lhd} {\ast}}} 

\newcommand \eexpan[1]{#1^{{\lhd} {\lhd}}}

\newcommand\expanorig[1]{(\expan #1, \orig)}

\newcommand \fgf U 
\newcommand \frf F 
\newcommand \exmon T 

\newcommand \isom \simeq 
\newcommand \emb \hookrightarrow 
\newcommand \id[1] {\mathrm{id}_{#1}} 

\newcommand \munit \eta 
\newcommand \mult \mu 
\newcommand \scata {\scat_{+}} 
\newcommand\End{\operatorname{End}}

\newcommand \eps[1] {#1^\epsilon} 
\newcommand \SCE {\partial^{-}} 
\newcommand\source\SCE
\newcommand \TGE {\partial^{+}} 
\newcommand\target\TGE
\newcommand \EPS {\partial^\epsilon} 

\newcommand \ind \sharp  
\newcommand \ppal[1] {|#1|} 
\newcommand \tp \overline 
\newcommand \bt \underline 

\newcommand \cone \Lambda 
\newcommand \Cone[1] {\cone(#1, \orig)} 
\newcommand \Bas \pi 
\newcommand \inc {\eta} 

\newcommand \orig o 
\newcommand \jok {{\ast}} 
\newcommand \econe \xi 
\newcommand \ec \hat 

\newcommand \ch[1] {\langle #1 \rangle} 
\newcommand \cch[1] {\ch{\! \ch{#1} \!}} 
\newcommand \shift[1] {\lceil #1 \rceil} 
\newcommand \sch[2] {\ch{#1, #2}} 
\newcommand \smp \ch 
\newcommand \schsmp[2] {\ch{#1, \! \smp{#2} \!}}  

\newcommand\ADC{\mathbf{ADC}}

\newcommand\ev[1]{\operatorname{ev}^{}_{#1}} 

\begin{document}

\begin{abstract}
The aim of this paper is to give an alternative construction of Street's
cosimplicial object of orientals, based on an idea of Burroni that orientals
are free algebras for some algebraic structure on strict \oo-categories.
More precisely, following Burroni, we define the notion of an expansion on an
\oo-category and we show that the forgetful functor from strict
\oo-categories endowed with an expansion to strict \oo-categories is
monadic. By iterating this monad starting from the empty \oo-category, we
get a cosimplicial object in strict \oo-categories. Our main contribution is
to show that this cosimplicial object is the cosimplicial objects of
orientals. To do so, we prove, using Steiner's theory of augmented directed
chain complexes, a general result for comparing polygraphs having same
generators and same linearized sources and targets.
\end{abstract}

\maketitle

\section*{Introduction}

The $n$-th oriental $\norient{n}$, introduced by Street in~\cite{StreetOrient}, is a (strict, globular) $n$-category shaped on the standard $n$-simplex.
More precisely, $\norient{n}$ is an $n$-category freely generated by a polygraph (or~computad) whose generating $k$-cells
correspond to the $k$-faces of the standard $n$-simplex.
Orientals organize themselves into a cosimplicial object into $\ooCat$, the category of (strict, globular) \oo-categories and (strict) \oo-functors,
that is, into a functor
\begin{displaymath}
\orient:\scat\to\ocat,
\end{displaymath}
where $\scat$ denotes the simplex category. This cosimplicial object induces a functor
\begin{displaymath}
\nrv:\ocat\to \sset,
\end{displaymath}
called \ndef{Street's~nerve}, taking each \oo-category $C$ to the simplicial set
\begin{displaymath}
 \nrv C: \deltan n\mapsto \homset{\ocat}{\norient{n}}{C}.
\end{displaymath}
The original motivation of Street was to define a cohomology with
coefficients in an \oo-category.

The combinatorics involved in \cite{StreetOrient} is notoriously hard.  This
led Street to extract in~\cite{StreetParComp} (see also
\cite{StreetParCompCorr}) the essential properties making it work. This was
formalized in his notion of a \ndef{parity complex}.
Using~\cite{StreetParComp}, the $n$-th oriental becomes the $n$-category
associated to a simple structure, the parity complex given by the faces of the
standard $n$-simplex, all the difficulty being now
hidden in the general machinery of parity complexes. In the same paper, he
defined a join construction for parity complexes, leading to an inductive
construction of the orientals, that is, a construction of $\norient{n+1}$
from $\norient{n}$.

Alternative definitions of the orientals were given by various people.
Burroni proposed during a presentation~\cite{BurOxford} an
inductive definition with explicit formulas but he didn't compare his
definition to Street's one. A short summary of this work was published a few
years later~\cite{BurOrientCahiers}. A similar approach was taken
independently by Buckley and Garner in \cite{BucGar} who did compare their
definition to Street's one. Another definition was given by
Steiner using his theory of augmented directed complexes \cite{Steiner}
(see \cite{SteinerOrient} for a comparison of the two definitions). Finally,
the first author and Maltsiniotis defined in \cite{AraMaltsiJoint} a join
construction for strict \oo-categories and showed that the cosimplicial
object of orientals is induced by the unique monoid structure for this join
supported by  the terminal \oo-category.

One of the drawbacks of the inductive definitions of the orientals is that they
don't give for free a cosimplicial object, except if one can show that the
iterated construction is equipped with the structure of a monad. It
is claimed in \cite{BurOxford} and \cite{BurOrientCahiers} that it is indeed
the case but this is far from obvious from the defining formulas. Burroni
gave a beautiful solution to this difficulty in a
draft~\cite{BurOrientPreprint} that was meant to be the extended version of
\cite{BurOrientCahiers}: the iterated construction is the monad
corresponding to some explicit algebraic structure on \oo-categories
that he called ``\oo-initial''.

\asterism

The purpose of this paper is two-fold. First, we give a formal, complete
account of Burroni's ideas on orientals --- up to now only circulated as short
papers, preprints and presented in talks. Second, we show that Burroni's
definition is equivalent to Street's one. For this purpose, we prove a
general result for comparing polygraphs with same generators and same
linearized sources and targets using Steiner's theory of augmented directed
complexes.

Let us explain in more details Burroni's construction. If $C$ is an
\oo-category, an \ndef{expansion} on~$C$ (called an ``\oo-initial
structure'' in \cite{BurOrientPreprint}, and a contraction in
\cite{AraMaltsiCondE}, where the notion was introduced independently)
consists of a $0$-cell $\orig$ of $C$, called the \ndef{origin}, and a
directed homotopy, that is, an oplax natural transformation,
from the constant \oo-functor $\orig : C \to C$
to the identity \oo-functor $\id C : C \to C$, satisfying some degeneracy
conditions. When $C$ is a category (seen as an \oo-category with only
trivial cells from dimension $2$ on), the possible origins for an expansion
on $C$ are precisely the initial objects of $C$. In general, the origin of
an expansion should be thought of as an \oo-initial object (hence Burroni's
terminology). By abstract nonsense, the forgetful functor $\fgf : \expcat
\to \ocat$, where $\expcat$ stands for the category of \oo-categories
endowed with an expansion, admits a left adjoint and we thus get a monad $T
: \ocat \to \ocat$. This monad induces a cosimplicial object $\orient :
\scat \to \ocat$ defined on objects by $\norient{n} = T^{n+1}(\emptyset)$,
where $\emptyset$ is the empty \oo-category. This is the definition of
orientals given in~\cite{BurOrientPreprint}.

\asterism

The paper is organized as follows: Section~\ref{basic} recalls the basic
definitions about $\omega$\nbd-categories and polygraphs, and sets related
notations. We particularly stress the role of the endofunctor of cylinders
in $\ocat$ first introduced in \cite{MetResPol}. Section~\ref{monad}
contains the main definition of the paper, namely the notion of {\em
$\omega$-category with expansion}, and introduces the associated adjunction
between $\expcat$ and $\ocat$, leading to an abstract, very compact
definition of the orientals. We then give an explicit description of the
resulting monad when applied to an $\omega$-category freely generated by
a polygraph. In particular, we get that our orientals are freely generated
by polygraphs. In Section~\ref{sec:calculus} we give a refined description
of the combinatorics of these objects by means of a convenient notation, the
{\em oriental calculus}. Section~\ref{comparison} finally establishes that
our orientals coincide with those originally defined by Street.  To do so,
we prove, using Steiner's theory of augmented directed complexes, a general
result for comparing polygraphs having same generators and same linearized
sources and targets.

\subsection*{Acknowledgement}

We warmly thank Albert Burroni for having shared with us over the
years, during lengthy conversations, his numerous insights and a
categorical wisdom not found in books.

\section{Basic notions on higher dimensional categories} \label{basic}

This section briefly recalls the basic definitions concerning (strict, globular) \oo-categories
and fixes the notations to be used throughout this work. 

\subsection{Globular sets and higher dimensional categories}

\begin{paragr}
We write $\oglob$ for the \ndef{category of globular sets}:
\begin{itemize}
\item A \ndef{globular set} $C$ is an infinite sequence of sets $C_0, C_1, C_2, \ldots$
together with infinitely many \ndef{source} maps $\SCE$ and \ndef{target}
maps $\TGE$
\[
\begin{xy}
\xymatrix{C_0 & C_1 \doubl \SCE \TGE & C_2 \doubl \SCE \TGE & \doubl \SCE\TGE \cdots}
\end{xy}
\]
satisfying the \ndef{globular conditions}:
\[
\SCE \SCE = \SCE \TGE, \qquad
\TGE \SCE = \TGE \TGE.
\]
\item A \ndef{globular morphism} $f : C \to D$ is an infinite sequence of maps
\[
f_0 : C_0 \to D_0, f_1 : C_1 \to D_1, f_2 : C_2 \to D_2, \ldots
\]
commuting to the above maps, that is, making the  diagram
\[
\begin{xy}
  \xymatrix{C_0 \ar[d]_{f_0}	& C_1 \doubl \SCE \TGE \ar[d]_{f_1}		& C_2 \doubl \SCE \TGE \ar[d]_{f_2}		& \cdots \doubl{\SCE}{\TGE} \\
		 D_0				& D_1\doubl \SCE \TGE				& D_2 \doubl \SCE \TGE				&
   \cdots \doubl{\SCE}{\TGE}}
\end{xy}  
\] 
commute, in the sense that $f_n \EPS = \EPS f_{n+1}$ for $\e = \pm$ and $n
\ge 0$. Whenever $x$ belongs to~$C_n$, we shall simply write $f(x)$ for
$f_n(x)$.
\end{itemize}
\end{paragr}

\begin{paragr}
An element $x$ of $C_n$ is called a \ndef{cell of dimension} $n$, or simply
an $n$-\ndef{cell}, \ndef{in} $C$.
\begin{itemize}
\item If $n > 0$, its \ndef{source} $\sce x= \SCE x$ and its \ndef{target} $\tge x= \TGE x$
are $(n{-}1)$-cells in $C$.
\item For $i \leq  n$, its $i$-\ndef{source} (resp.~its $i$-\ndef{target})
is the $i$-cell $\EPS_i x = \EPS \! \cdots \EPS x$, where $\e$ stands
for~$-$ (resp.~for $+$) and $\EPS$ is applied $n-i$ times to $x$. We shall
also write $\eps x_i$ for $\EPS_i x$. In~particular, we get $\eps x_n = x$.
\end{itemize}
If $p < n$, to make the $i$-source and $i$-target of the $n$-cell $x$
explicit for $p \leq i < n$, we write
\[
x : \sce x_{n-1} \to \tge x_{n-1} : \cdots : \sce x_p \to \tge x_p.
\]
We say that two $n$-cells $x$ and $y$ of $C$ are \ndef{parallel} if, either $n =
0$, or $n > 0$ and $x$ and $y$ have same source and same target.
\end{paragr}

\begin{examples}
We get $x : \sce x_0 \to \tge x_0$ if $x$ is a 1-cell, and $x : \sce x_1 \to \tge x_1 : \sce x_0 \to \tge x_0$ if $x$~is~a~2-cell.
\[
\begin{xy}
    \xymatrix{\scriptstyle{\sce x_0}\cdot\ar[r]^-{x} & \cdot\ \scriptstyle{\tge x_0}}
  \end{xy}
\hspace{5em}
  \begin{xy}
    \xymatrix{\scriptstyle{\sce x_0}\cdot
    \rrtwocell^{\sce x_1}_{\tge x_1}{x}
    & &\cdot\ \scriptstyle{\tge x_0}}
  \end{xy}
\]
\end{examples}

\begin{paragr}
By restriction to finite sequences $C_0,\dots,C_n$, we get the \ndef{category $\nglob n$ of $n$-globular sets}.
In particular, $\nglob 0$ is $\Set$, the category of sets, and $\nglob 1$ is the category of (directed) graphs.
Note that there is an obvious \ndef{truncation functor} from $\oglob$ to $\nglob n$,
mapping $C$ to the $n$-globular set $\trunc n C$ obtained by removing all cells of dimension $> n$.
\end{paragr}

\begin{paragr}
We write $\ocat$ for the \ndef{category of \oo-categories}:
\begin{itemize}
\item An \oo-\ndef{category} is a globular set $C$, together with \ndef{compositions} and \ndef{units}
satisfying the~laws of \ndef{associativity}, \ndef{unit}, \ndef{interchange} and \ndef{functoriality of units}.
\item An \oo-\ndef{functor} is a globular morphism $f : C \to D$ preserving compositions and units.
\end{itemize}
\end{paragr}

\begin{paragr}
For $n > p$ and for any $n$-cells $x, y$ such that $\tge x_p = \sce y_p$ in an \oo-category $C$,
we get an $n$-cell $z = y \comp p x$, called $p$-\ndef{composition of $y$
and $x$}, with the following iterated sources and targets:
\[
\eps z_i = \eps x_i = \eps y_i \mbox{ for } i < p, \qquad
\sce z_p = \sce x_p, \qquad
\tge z_p = \tge y_p, \qquad
\eps z_i = \eps y_i \comp p \eps x_i \mbox{ for } p < i < n.
\]
We shall omit parentheses by giving priority to the lowest dimensional composition, namely:
\[
z \comp p y \comp q x =
\left\{
\begin{array}{ll}
(z \comp p y) \comp q x & \mbox{if } p \leq q, \myspace \\
z \comp p (y \comp q x) & \mbox{if } p \geq q.
\end{array}
\right.
\]
By associativity, both conventions are indeed compatible in case $p = q$.
\end{paragr}

\begin{paragr}
For any $p$-cell $u$ in an \oo-category $C$, we get a $(p{+}1)$-cell $\unit u : u \to u$, called $p$-\ndef{unit on} $u$.
By iterating this operator $i$ times for $i > 0$, we get the following $p$-\ndef{unit of dimension $p+i$ on} $u$:
\[
\Unit i u : \Unit{i - 1} u \to \Unit{i - 1} u : \cdots : \unit u \to \unit u : u \to u.
\]
By the law of functoriality of units, any $p$-cell $u$ can be identified with the $n$-cell $\Unit{n - p} u$ for $n > p$.
In fact, we shall not identify them, but we shall use the following abbreviations for $n > p > q$,
for any $n$-cell $x$ and any $p$-cell $u$ such that $\tge u_q = \sce x_q$ (resp.~$\tge x_q = \sce u_q$):
\[
x \comp q u = x \comp q \Unit{n - p} u, \qquad
u \comp q x = \Unit {n - p} u \comp q x.
\]
\end{paragr}

\begin{example}
For any 2-cell $x$ and any 1-cell $u$ such that $\tge u_0 = \sce x_0$, we get
\[
x \comp 0 u = x \comp 0 \unit u : \sce x_1 \comp 0 u \to \tge x_1 \comp 0 u : \sce u_0 \to \tge x_0.
\]
\[
  \begin{xy}
    \xymatrix{\scriptstyle{\sce u_0}\cdot
    \ar[r]^{u}&\underset{\makebox[0pt]{$\scriptstyle{\tge u_0 = \, \sce x_0}$}}{\cdot}\rrtwocell^{\sce x_1}_{\tge x_1}{x} & &\cdot\ \scriptstyle{\tge x_0}}
  \end{xy}
\]
\end{example}

\begin{paragr}
By using $n$-globular sets instead of globular sets, we get the \ndef{category $\ncat n$ of $n$-categories}.
In particular, $\ncat 0$ is $\Set$ and $\ncat 1$ is $\Cat$, the category of
small categories.

The truncation functor $C \mapsto \trunc{n}C$ from globular sets to
$n$-globular sets extends to a \ndef{truncation functor} from $\ocat$ to
$\ncat n$, which we will denote in the same way. This functor admits a left
adjoint mapping any $n$-category $C$ to the \oo-category obtained by adding
an $n$-unit $\Unit i u$ of dimension $n + i$ for $i > 0$ and for each
$n$-cell $u$ in $C$. This \ndef{canonical embedding} yields an equivalence
between $\ncat n$ and the full subcategory of $\ocat$ whose objects only
have unit cells beyond dimension $n$.  In other~words, any $n$-category, and
in particular any set, can be seen as an \oo-category.
\end{paragr}

\begin{paragr}
The category $\ocat$ is complete and cocomplete.
In particular, we get the following two \oo-categories:
\begin{itemize}
\item the \ndef{initial \oo-category} is the empty set $\emptyset$, which has no cell,
\item the \ndef{terminal \oo-category} is the singleton set $\sing = \set \orig$, which has a single 0-cell $\orig$,
and a~single $n$-cell $\Unit n \orig$ for each $n > 0$.
\end{itemize}
\end{paragr}

\begin{paragr}
The category $\ocat$ is the limit of the following diagram of categories, where arrows are truncation functors:
  \[
    \begin{xy} 
      \xymatrix{\ncat 0 & \ncat 1 \ar[l] & \ncat 2 \ar[l] & \cdots \ar[l]}
    \end{xy}
  \]
Moreover, the category $\ncat{n+1}$ is \ndef{enriched over} $\ncat n$, and likewise, $\ocat$ is enriched over itself.
For any \oo-category $C$ and for any 0-cells $u, v$ in $C$, we get indeed another \oo-category $\homset C u v$:
\begin{itemize}
\item An $n$-cell in $\homset C u v$ is an $(n{+}1)$-cell $x$ in $C$ such that $\sce x_0 = u$ and $\tge x_0 = v$.
\item The $p$-composition and $p$-units in $\homset C u v$ are the $(p{+}1)$-composition and $(p{+}1)$-units in~$C$.
\end{itemize}
\end{paragr}

\subsection{Polygraphs}

The forgetful functor from $\ocat$ to $\oglob$ has a left adjoint, yielding a notion of~\oo-\ndef{category freely generated by a globular set}.
Here, we describe a more general notion of \oo-\ndef{category freely generated by a polygraph}, or \ndef{computad},
introduced independently in \cite[Section 4]{StreetOrient} and \cite{BurHighdim}.

\begin{paragr}
Consider the following commutative diagram of categories, where the horizontal arrows are forgetful functors, and the vertical ones are truncation functors:
\[
  \begin{xy}
    \xymatrix{\ncat{n+1} \ar[r]\ar[d]& \nglob{n+1}\ar[d]\\
                     \ncat n \ar[r] & \nglob n}
  \end{xy}
\]
We get a functor $\Right n: \ncat{n+1} \to \ncatpl n$, where $\ncatpl n$ is defined by the following pullback~square:
\[
 \begin{xy}
    \xymatrix{\ncatpl n \ar[r]\ar[d]\ar@{}@<-3pt>[rd]|(0.2){\displaystyle{\lrcorner}} & \nglob{n+1}\ar[d]\\
                     \ncat n \ar[r] & \nglob n}
  \end{xy}
\]
It happens that this functor has a left adjoint $\Left n : \ncatpl n \to \ncat{n+1}$.
See \cite{BurHighdim} or~\cite{MetCofib}.
\end{paragr}

\begin{paragr}
More concretely, an object of $\ncatpl n$ is a pair $(C,S_{n+1})$, where $C$ is an $n$-category
and $S_{n+1}$ is a set of $(n{+}1)$-\ndef{generators}, together with maps $\SCE, \TGE : S_{n+1} \to C_n$
satisfying the globular conditions in case $n > 0$.
The functor $\Left n$ maps this object to the $(n{+}1)$-category
\[
  \begin{xy}
    \xymatrix{C_0 & \cdots \doubl{\SCE}{\TGE} & C_n \doubl{\SCE}{\TGE} & \free S_{n+1}, \doubl{\SCE}{\TGE}}
  \end{xy}
\]
where $\free S_{n+1}$ consists of formal compositions of $(n{+}1)$-generators and $n$-units,
quotiented by the laws of associativity, unit, interchange and functoriality of units.

By construction, the map $S_{n+1} \emb \free S_{n+1}$, which can be shown to
be an injection,  commutes to the source and target maps, and the above
$(n{+}1)$-category $C' = \Left n(C,S_{n+1})$ has the following universal
property:
\end{paragr}

\begin{lemma}\label{lemma:universal}
Consider some $(n{+}1)$-category $D$ and some $n$-functor $f : C \to \trunc
n D$.  Then any map $e_{n+1} : S_{n+1} \to D_{n+1}$ such that $\EPS e_{n+1}
= f_n \EPS$ for $\e = \pm$
extends to a unique map $f_{n+1} : \free S_{n+1} \to D_{n+1}$ such that $f_0, \ldots, f_{n+1}$ form an $(n{+}1)$-functor from $C'$ to $D$.
\[
\begin{xy}
  \xymatrix{C_0\ar[d]_{f_0}		& \cdots \doubl{\SCE}{\TGE} & C_n \doubl{\SCE}{\TGE} \ar[d]_{f_n}
	& \free S_{n+1} \doubl \SCE \TGE \ar@{.>}[d]_{f_{n+1}}	& \, S_{n+1}\ar[dl]^{e_{n+1}}\ar@{_{(}->}[l] \\
		 D_0				& \cdots \doubl{\SCE}{\TGE} & D_n \doubl{\SCE}{\TGE}
	& D_{n+1} \doubl \SCE \TGE						&}
\end{xy}  
\] 
\end{lemma}

\begin{paragr}
By induction on $n$, we define the \ndef{category $\npol n$ of $n$-polygraphs}, together with a functor
$\Free n : \npol n \to \ncat n$ mapping any $n$-polygraph $S$ to the \ndef{free $n$-category $\free S \!$ generated by} $S$:
\begin{itemize}
\item The category $\npol 0$ is $\ncat 0$, that is $\Set$, and $\Free 0 : \Set \to \Set$ is the identity functor.
\item Suppose that the category $\npol n$ and the functor $\Free n :
  \npol n\to \ncat n$ have been defined. Then
the category $\npol{n+1}$ is given by the pullback square
\[
  \begin{xy}
    \xymatrix{\npol{n+1}\ar[r]\ar[d]\ar@{}@<-3pt>[rd]|(0.2){\displaystyle{\lrcorner}} & \ncatpl {n}\ar[d]\\
                     \npol n \ar[r]_{\Free n} & \ncat n}
  \end{xy}
\]
and $\Free {n+1} : \npol{n+1} \to \ncat{n+1}$ is the composition of $\Left n : \ncatpl n \to \ncat{n+1}$ by the top~arrow.
\end{itemize}
In particular, $\npol 1$ is $\ncatpl 0$, that is $\nglob 1$, and $\Free 1(S) = \free S$ is the free category generated by $S$.
\end{paragr}

\begin{definition}[polygraphs] \ %
\noindent
The \ndef{category $\opol$ of polygraphs} is the limit of the following diagram,
where each arrow is the truncation functor given by the previous pullback square:
  \[
    \begin{xy}
      \xymatrix{\npol 0 & \npol 1 \ar[l] & \npol 2 \ar[l] & \cdots \ar[l]}
    \end{xy}
  \]
The functors $\Free n$ induce a functor $\oFree : \opol \to \ocat$ mapping $S$ to the \ndef{free \oo-category $\free S \!$ generated by}~$S$.
\end{definition}

\begin{paragr}
More concretely:
\begin{itemize}
\item A polygraph $S$ is given by an infinite diagram of the form
  \[
    \begin{xy}
      \xymatrix{S_0 \ar@{_{(}->}[d]	& S_1 \ar@{_{(}->}[d] \DOUBLD	& S_2 \ar@{_{(}->}[d] \DOUBLD	& \cdots \DOUBLD	 \\
			\free{S_0}			& \free{S_1} \DOUBL			& \free{S_2} \DOUBL			& \cdots \DOUBL	}
    \end{xy}
  \]
where $S_i$ is a set of $i$-\ndef{generators}, and the bottom row displays $\free S$, starting from $\free S_0 = S_0$.
\item A morphism $f : S \to T$ is given by an infinite sequence of maps
\[
f_0 : S_0 \to T_0, \; f_1 : S_1 \to T_1, \; f_2 : S_2 \to T_2, \ldots
\]
compatible with sources and targets so that they induce maps
\[
\free f_0 : \free S_0 \to \free T_0, \free f_1 : \free S_1 \to \free T_1, \free f_2 : \free S_2 \to \free T_2, \ldots
\]
defining an \oo-functor $\free f \! : \free S \! \to \free T \!$.
This means that $\oFree (f) = \free f$ is \ndef{rigid}: it preserves generators.
\end{itemize}
In case $S_i = \emptyset$ for $i > n$, the polygraph $S$ is in fact an $n$-polygraph, and $\free S$ is an $n$-category.
This \ndef{canonical embedding} of $\npol n$ into $\opol$ is the left adjoint of the obvious truncation functor.
\end{paragr}

\begin{examples}\ %
\begin{itemize}
\item If $S$ consists of a single 0-generator $\orig$, then $\free S$ is the singleton set $\sing = \set \orig$.
\item If $S$ consists of two 0-generators $\orig, \orig'$ and a single 1-generator $\sigma : \orig \to \orig'$,
then $\free S$ consists of the generators and the 0-units $\unit \orig : \orig \to \orig$, $\unit{\orig'} : \orig' \to \orig'$.
\item If $S$ consists of three 0-generators $\orig$, $\orig'$, $\orig''$,
three 1-generators $\sigma : \orig \to \orig'$, $\sigma' : \orig' \to \orig''$, $\sigma'' : \orig \to \orig''$,
and a single 2-generator $\tau : \sigma'' \to \sigma' \comp 0 \sigma$,
then $\free S$ consists of the generators, the 0-composition $\sigma' \comp 0 \sigma : \orig \to \orig''$,
the 0-units $\unit \orig : \orig \to \orig, \unit{\orig'} : \orig' \to \orig'$, $\unit{\orig''} : \orig'' \to \orig''$,
the~1-units $\unit \sigma : \sigma \to \sigma, \unit{\sigma'} : \sigma' \to \sigma'$, $\unit{\sigma''} : \sigma'' \to \sigma''$,
$\unit{\sigma' \comp 0 \sigma} : \sigma' \comp 0 \sigma \to \sigma' \comp 0 \sigma$,
and the~0-units $\Unit 2 \orig : \unit \orig \to \unit \orig$, $\Unit 2 {\orig'} : \unit{\orig'} \to \unit{\orig'}$, $\Unit 2 {\orig''} : \unit{\orig''} \to \unit{\orig''}$ of dimension 2.
\item If $S$ consists of a single 0-generator $\orig$ and a single 1-generator $\sigma : \orig \to \orig$,
then $\free S$ consists of the 0-generator, and infinitely many 1-cells, which are of the form
\[
\sigma^0 = \unit \orig : \orig \to \orig, \qquad
\sigma^i = \sigma \comp 0 \cdots \comp 0 \sigma : \orig \to \orig \mbox{ for } i > 0.
\]
\item If $S$ consists of a single 0-generator $\orig$ and a single 2-generator $\tau : \unit \orig \to \unit \orig$,
then $\free S$ consists of the 0-generator, the 0-unit $\unit \orig : \orig \to \orig$, and infinitely many 2-cells, which are of the form
\[
\tau^0 = \Unit 2 \orig : \unit \orig \to \unit \orig : \orig \to \orig, \qquad
\tau^i = \tau \comp 0 \cdots \comp 0 \tau = \tau \comp 1 \cdots \comp 1 \tau : \unit \orig \to \unit \orig : \orig \to \orig \mbox{ for } i > 0.
\]
\end{itemize}
\[
  \begin{xy}
    \xymatrix{ &&&&&&&\overset{\orig'}{\cdot}\ar[rd]^{\sigma'}& &&&&\\
\overset{\orig}{\cdot} & &\overset{\orig}{\cdot}\ar[rr]_{\sigma} & 
& \overset{\orig'}{\cdot} && \underset{\orig}{\cdot}
\ar[rr]_{\sigma''}\ar[ru]^{\sigma} &
\ar@{=>}(81,-15)*{};(81,-7)*{}_{\tau}
&
\underset{\orig''}{\cdot}&&\underset{\orig}{\cdot}\ar@(ul,ur)^{\sigma}&&
\underset{\orig}{\cdot}\ar@{=>}@(ul,ur)^{\tau}}
  \end{xy}
\]
\end{examples}

\begin{remarks}
The first three examples are (isomorphic to) the first three \ndef{orientals} $\norient 0, \norient 1, \norient 2$.
The last two ones are (isomorphic to) the additive monoid $\NN$,
respectively seen as a $1$-category with a single $0$-cell and as a
$2$-category with a single $0$-cell and a single $1$-cell.
\end{remarks}

\subsection{Cylinders and oplax transformations}

Here, we recall the construction of the \ndef{endofunctor of small cylinders},
and the resulting notion of \ndef{oplax transformation} between two \oo-functors,
called \ndef{homotopy} in \cite{MetResPol} and~\cite{LafMetPolRes}.

For that purpose, we first define the set $\nCyl n C$ of \ndef{$n$-cylinders} in an \oo-category~$C$,
together with two maps $\SCE, \TGE : \nCyl n C \to \nCyl {n-1} C$ in case $n > 0$.
We can then define a structure of \oo-category on the corresponding globular set, which gives the expected endofunctor.

\begin{definition}[cylinders]  \label{def:cyl} \ %

\noindent
If $x, y$ are $n$-cells in $C$, the notion of \ndef{$n$-cylinder} $\alpha : x \cto y$ is given inductively:
\begin{itemize}
\item If $n = 0$, then $\alpha : x \cto y$ consists of a single 1-cell $\alpha_0 : x \to y$ in $C$.
\item If $n > 0$, then $\alpha : x \cto y$ consists of two 1-cells $\sce \alpha_0 : \sce x_0 \to \sce y_0$ and $\tge \alpha_0 : \tge x_0 \to \tge y_0$ in $C$,
together with an $(n{-}1)$-cylinder $\ind \alpha : \tge \alpha_0 \comp 0 x \cto y \comp 0 \sce \alpha_0$ in $\homset C{\sce x_0}{\tge y_0}$.
\end{itemize}
In case $n > 0$, we also get two $(n{-}1)$-cylinders $\SCE \alpha : \SCE x \cto \SCE y$ and $\TGE \alpha : \TGE x \cto \TGE y$
which are given inductively:
\begin{itemize}
\item If $n = 1$, then $\EPS \alpha : \eps x_0 \cto \eps y_0$ is given by the 1-cell $\eps \alpha_0 : \eps x_0 \to \eps y_0$.
\item If $n > 1$, then $\EPS \alpha : \eps x_{n-1} \cto \eps y_{n-1}$ is given by
the 1-cells $\sce \alpha_0 : \sce x_0 \to \sce y_0$ and $\tge \alpha_0 : \tge x_0 \to \tge y_0$,
together with the $(n{-}2)$-cylinder $\ind \EPS \alpha = \EPS \ind \alpha : \tge \alpha_0 \comp 0 \eps x_{n-1} \cto \eps y_{n-1} \comp 0 \sce \alpha_0$
in $\homset C{\sce x_0}{\tge y_0}$.
\end{itemize}
If $\alpha : x \cto y$ is such an $n$-cylinder, we write $\tp \alpha$ and $\bt \alpha$ for the $n$-cells $x$ and $y$ respectively.
\end{definition}

\begin{paragr} \label{paragr:conc}
More concretely, an $n$-cylinder $\alpha : x \cto y$ in $C$ is given by a finite sequence of cells
\[
\sce \alpha_0, \tge \alpha_0, \ldots, \sce \alpha_{n-1}, \tge \alpha_{n-1}, \alpha_n \mbox{ in } C,
\]
where the \ndef{auxiliary cells} $\eps \alpha_i$ have dimension $i + 1$,
the \ndef{principal cell} $\ppal \alpha = \alpha_n$ has dimension~$n + 1$,
and their sources and targets are given as follows:
\[
\begin{array}{c}
\eps \alpha_i : \tge \alpha_{i-1} \comp {i-1} \cdots \comp 1 \tge \alpha_0 \comp 0 \eps x_i \to
\eps y_i \comp 0 \sce \alpha_0 \comp 1 \cdots \comp {i-1} \sce \alpha_{i-1} 
\mbox{ for } i < n, \myspace \\
\ppal \alpha = \alpha_n : \tge \alpha_{n-1} \comp {n-1} \cdots \comp 1 \tge \alpha_0 \comp 0 x \to
y \comp 0 \sce \alpha_0 \comp 1 \cdots \comp {n-1} \sce \alpha_{n-1}.
\end{array}  
\]
For any $i < n$, the $i$-cylinder $\EPS_i \alpha : \eps x_i \cto \eps y_i$ is given by the following sequence of cells:
\[
\sce \alpha_0, \tge \alpha_0, \ldots, \sce \alpha_{i-1}, \tge \alpha_{i-1}, \eps \alpha_i.
\]
\end{paragr}

\begin{remark}
Beware of a slight discrepancy in our notations:
\begin{itemize}
\item If $x$ is a cell, then $\eps x_i$ stands for the $i$-cell $\EPS_i x$.
\item If $\alpha$ is a cylinder, then $\eps \alpha_i$ does not stand for the
$i$-cylinder $\EPS_i \alpha$, but for its \ndef{principal
cell}~$\ppal{\EPS_i \alpha}$.
\end{itemize}
\end{remark}

\begin{examples}\ %
\begin{itemize}
\item If $x, y$ are 0-cells, a 0-cylinder $\alpha : x \cto y$ is given by the 1-cell $\ppal \alpha = \alpha_0 : x \to y$.
\item If $x, y$ are 1-cells, a 1-cylinder $\alpha : x \cto y$ is given by the 1-cells $\eps \alpha_0 : \eps x_0 \to \eps y_0$
and the 2-cell $\ppal \alpha = \alpha_1 : \tge \alpha_0 \comp 0 x \to y \comp 0 \sce \alpha_0$.
\item If $x, y$ are 2-cells, a 2-cylinder $\alpha : x \cto y$ is given by the 1-cells $\eps \alpha_0 : \eps x_0 \to \eps y_0$,
the 2-cells $\eps \alpha_1 : \tge \alpha_0 \comp 0 \eps x_1 \to \eps y_1 \comp 0 \sce \alpha_0$,
and the 3-cell $\ppal \alpha = \alpha_2 : \tge \alpha_1 \comp 1 \tge \alpha_0 \comp 0 x \to y \comp 0 \sce \alpha_0 \comp 1 \sce \alpha_1$.
\end{itemize}
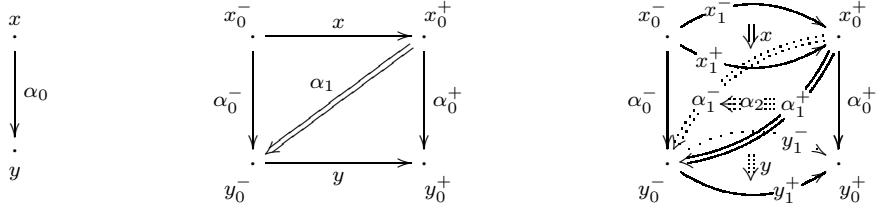
\[
  \begin{xy}
    \xymatrix{\overset{x}{\cdot} \ar[d]^{\alpha_0}\\
                       \underset{y}{\cdot}}
  \end{xy}
\hspace{5em}
  \begin{xy}
    \xymatrix{\overset{\makebox[0pt]{$\scriptstyle{\sce x_0}\ \ \
          $}}{\cdot} 
         \ar[rr]^{x}
         \ar[d]_{\sce \alpha_0} &&
      \overset{\makebox[0pt]{$\ \ \ \scriptstyle{\tge x_0}$}}{\cdot} 
      \ar[d]^{\tge \alpha_0}
       \ar@{=>}[lld]_{\alpha_1}\\
\underset{\makebox[0pt]{$\scriptstyle{\sce y_0}\ \ \ $}}{\cdot}
        \ar[rr]_y
& &\underset{\makebox[0pt]{$\ \ \ \scriptstyle{\tge y_0}$}}{\cdot}  }
  \end{xy}
\hspace{5em}
  \begin{xy}
    \xymatrix{\ar@3{.>}(14,-9)*{};(7,-9)*{}|(0.4){\alpha_2}
                      \ar@2{->}(11,1.5)*{};(11,-1.5)*{}^{x}
                       \ar@2{.>}(11,-16)*{};(11,-19)*{}^{y}
                     \overset{\makebox[0pt]{$\scriptstyle{\sce x_0}\ \ \ $}}{\cdot}
                       \ar@/^1pc/[rr]|(.3){\sce x_1}
                      \ar@/_1pc/[rr]|(.25){\tge x_1}
                      \ar[d]_{\sce \alpha_0}
      &&\overset{\makebox[0pt]{$\ \ \
          \scriptstyle{\tge x_0}$}}{\cdot}
                      \ar[d]^{\tge \alpha_0}
                       \ar@2{->}@/^1pc/[lld]|(.35){\tge \alpha_1}
                       \ar[d]\ar@2{.>}@/_1pc/[lld]|(.67){\sce \alpha_1}\\
           \underset{\makebox[0pt]{$\scriptstyle{\sce y_0}\ 
                            \ \ $}}{\cdot} 
                        \ar@{.>}@/^1pc/[rr]|(.75){\sce y_1}
                         \ar@/_1pc/[rr]|(.7){\tge y_1} 
      && \underset{\makebox[0pt]{$\ \ \ \scriptstyle{\tge y_0}$}}{\cdot}}
  \end{xy}
\]
\end{examples}

\begin{paragr}
If we write $\nCyl n C$ for the set of $n$-cylinders in $C$, we get the following globular set $\Cyl C$: 
\[
\begin{xy}
\xymatrix{\nCyl 0 C & \nCyl 1 C \doubl \SCE \TGE & \nCyl 2 C \doubl \SCE \TGE & \doubl \SCE \TGE \cdots}
\end{xy}
\]
\end{paragr}

The globular set $\Cyl C$ supports a structure of \oo-category we
describe below (see also~\cite{LafMetPolRes,LafMetWorFolk}).

\begin{definition}[\oo-category of small cylinders] \label{def:comp} \ %

\noindent
The \oo-\ndef{category of small cylinders} in $C$ is $\Cyl C$ endowed with the following operations:
\begin{itemize}
\item If $n > p$, if $x, y, z, t$ are $n$-cells such that $\tge x_p = \sce z_p$ and $\tge y_p = \sce t_p$,
and if $\alpha : x \cto y$, $\beta : z \cto t$ are $n$-cylinders such that $\TGE_p \alpha = \SCE_p \beta$,
the \ndef{$p$-composition} $\gamma = \beta \comp p \alpha : z \comp p x \cto t \comp p y$ is the $n$-cylinder given by the following cells:
\end{itemize}
\[
\begin{array}{c}
\eps \gamma_i = \eps \alpha_i = \eps \beta_i \mbox { for } i < p, \qquad
\sce \gamma_p = \sce \alpha_p, \qquad
\tge \gamma_p = \tge \beta_p, \myspace \\
\eps \gamma_{p+1} = (\eps t_{p+1} \! \comp 0 \sce \alpha_0 \! \comp 1 \cdots \comp{p-1} \sce \alpha_{p-1} \! \comp p \eps \alpha_{p+1}) \comp{p+1}
(\eps \beta_{p+1} \! \comp p \tge \beta_{p-1} \! \comp{p-1} \cdots \comp 1 \tge \beta_0 \! \comp 0 \eps x_{p+1}) \mbox { if } p + 1 < n, \myspace \\
\eps \gamma_i = (\tge t_{p+1} \! \comp 0 \sce \alpha_0 \! \comp 1 \cdots \comp{p-1} \sce \alpha_{p-1} \! \comp p \eps \alpha_i) \comp{p+1}
(\eps \beta_i \! \comp p \tge \beta_{p-1} \! \comp{p-1} \cdots \comp 1 \tge \beta_0 \! \comp 0 \sce x_{p+1}) \mbox { for } p + 1 < i < n, \myspace \\
\ppal \gamma = \gamma_n = (\tge t_{p+1} \comp 0 \sce \alpha_0 \comp 1 \cdots \comp{p-1} \sce \alpha_{p-1} \comp p \alpha_n) \comp{p+1}
(\beta_n \comp p \tge \beta_{p-1} \comp{p-1} \cdots \comp 1 \tge \beta_0 \comp 0 \sce x_{p+1}).
\end{array}
\]
\begin{itemize}
\item If $\alpha : x \cto y$ is a $p$-cylinder, the \ndef{$p$-unit} $\unit \alpha : \unit x \cto \unit y$ is given by the following cells:
\[
\sce \alpha_0, \tge \alpha_0, \ldots, \sce \alpha_{p-1}, \tge \alpha_{p-1}, \alpha_p, \alpha_p, \unit{\alpha_p}.
\]
\end{itemize}
We refer to \cite[Appendix A]{MetResPol} for a proof that the axioms of (strict) \oo-categories hold.
\end{definition}

\goodbreak

\begin{examples}\ %
\begin{itemize}
\item If $x, y, z, t$ are 1-cells such that $\tge x_0 = \sce z_0$ and $\tge y_0 = \sce t_0$,
and if $\alpha : x \cto y, \beta : z \cto t$ are 1-cylinders such that $\tge \alpha_0 = \sce \beta_0$,
the 0-composition $\gamma = \beta \comp 0 \alpha : z \comp 0 x \cto t \comp 0 y$ is given by
the 1-cells $\sce \gamma_0 = \sce \alpha_0$ and $\tge \gamma_0 = \tge \beta_0$,
and the 2-cell $\ppal \gamma = \gamma_1 = (t \comp 0 \alpha_1) \comp 1 (\beta_1 \comp 0 x)$.
\item If $x, y, z, t$ are 2-cells such that $\tge x_0 = \sce z_0$ and $\tge y_0 = \sce t_0$,
and if $\alpha : x \cto y, \beta : z \cto t$ are 2-cylinders such that $\tge \alpha_0 = \sce \beta_0$,
the 0-composition $\gamma = \beta \comp 0 \alpha : z \comp 0 x \cto t \comp 0 y$ is given by
the 1-cells $\sce \gamma_0 = \sce \alpha_0$ and $\tge \gamma_0 = \tge \beta_0$,
the 2-cells $\eps \gamma_1 = (\eps t_1 \comp 0 \eps \alpha_1) \comp 1 (\eps \beta_1 \comp 0 \eps x_1)$,
and the 3-cell $\ppal \gamma = \gamma_2 = (\tge t_1 \comp 0 \alpha_2) \comp 1 (\beta_2 \comp 0 \sce x_1)$.
\item If $x, y, z, t$ are 2-cells such that $\tge x_1 = \sce z_1$ and $\tge y_1 = \sce t_1$,
and if $\alpha : x \cto y, \beta : z \cto t$ are 2-cylinders such that $\eps \alpha_0 = \eps \beta_0$ and $\tge \alpha_1 = \sce \beta_1$,
the 1-composition $\gamma = \beta \comp 1 \alpha : z \comp 1 x \cto t \comp 1 y$ is given by
the 1-cells $\eps \gamma_0 = \eps \alpha_0 = \eps \beta_0$,
the 2-cells $\sce \gamma_1 = \sce \alpha_1$ and $\tge \gamma_1 = \tge \beta_1$,
and the 3-cell $\ppal \gamma = \gamma_2 = (t \comp 0 \sce \alpha_0 \comp 1 \alpha_2) \comp 2 (\beta_2 \comp 1 \tge \beta_0 \comp 0 x)$.
\end{itemize}
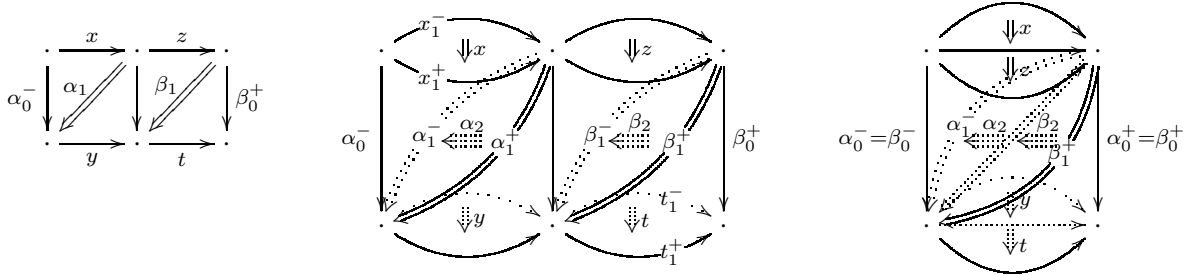
\[
  \begin{xy}
    \xymatrix{\cdot \ar[r]^x\ar[d]_{\sce \alpha_0} & \cdot\ar[r]^z\ar[d] \ar@2{->}[dl]_*-<4pt>{\scriptstyle{\alpha_1}} & 
    \cdot\ar[d]^{\tge \beta_0}\ar@2{->}[dl]_*-<4pt>{\scriptstyle{\beta_1}}\\
\cdot \ar[r]_y&  \cdot\ar[r]_t &  \cdot}
  \end{xy}
\hspace{2em}
  \begin{xy}
    \xymatrix{\cdot
                     \ar@3{.>}(13,-12)*{};(8,-12)*{}_(.2){\alpha_2}
                      \ar@2{->}(11,1.5)*{};(11,-1.5)*{}^{x}
                      \ar@2{.>}(11,-21)*{};(11,-24)*{}^{y}
                      \ar@3{.>}(35,-12)*{};(30,-12)*{}_(.2){\beta_2}
                      \ar@2{->}(33,1.5)*{};(33,-1.5)*{}^{z}
                      \ar@2{.>}(33,-21)*{};(33,-24)*{}^{t}
                       \ar@/^1pc/[rr]|(.3){\sce x_1}
                       \ar@/_1pc/[rr]|(.3){\tge x_1}\ar[dd]_{\sce \alpha_0}
      &&\cdot 
                       \ar@/^1pc/[rr]
                       \ar@/_1pc/[rr]\ar[dd]
                         \ar@2{->}@/^1pc/[lldd]|(.4){\tge \alpha_1}
                       \ar@2{.>}@/_1pc/[lldd]|(.6){\sce \alpha_1}
       && \cdot
                        \ar[dd]^{\tge \beta_0}
                       \ar@2{->}@/^1pc/[lldd]|(.4){\tge \beta_1}
                       \ar@2{.>}@/_1pc/[lldd]|(.6){\sce \beta_1}    \\
&&&&\\
                  \cdot 
                   \ar@{.>}@/^1pc/[rr]
                    \ar@/_1pc/[rr] 
      && \cdot  
                    \ar@{.>}@/^1pc/[rr]|(.7){\sce t_1}
                    \ar@/_1pc/[rr]|(.7){\tge t_1} && \cdot}
  \end{xy}
\hspace{2em}
  \begin{xy}
    \xymatrix{ \cdot
                        \ar@3{.>}(10,-12)*{};(5,-12)*{}_(.2){\alpha_2}
                         \ar@3{.>}(17,-12)*{};(12,-12)*{}_(.2){\beta_2}
                        \ar@2{->}(11,4)*{};(11,1)*{}^(.4){x}
                        \ar@2{->}(11,-1)*{};(11,-4)*{}^(.6){z}
                        \ar@2{.>}(11,-19)*{};(11,-22)*{}^(.4){y}
                        \ar@2{.>}(11,-24)*{};(11,-27)*{}^(.6){t}
                         \ar@/^{1.5pc}/[rr]
                        \ar[rr]\ar@/_{1.5pc}/[rr]
                         \ar[dd]_{\sce \alpha_0=\sce \beta_0}&
                          &\cdot \ar[dd]^{\tge \alpha_0=\tge \beta_0}
                             \ar@2{->}@/^{1.5pc}/[lldd]|(.4){\tge \beta_1}
                               \ar@2{.>}@/_{1.5pc}/[lldd]|(.6){\sce \alpha_1}
                                \ar@2{.>}[lldd]\\
&& \\
\cdot
                         \ar@{.>}@/^{1.5pc}/[rr]
                          \ar@{.>}[rr]
                           \ar@/_{1.5pc}/[rr]
&&\cdot}
  \end{xy}
\]
\begin{itemize}
\item If $\alpha : x \cto y$ is a 0-cylinder, the 0-unit $\unit \alpha : \unit x \cto \unit y$ is given by the cells $\alpha_0, \alpha_0, \unit{\alpha_0}$.
\item If $\alpha : x \cto y$ is a 1-cylinder, the 1-unit $\unit \alpha : \unit x \cto \unit y$ is given by the cells
$\sce \alpha_0, \tge \alpha_0, \alpha_1, \alpha_1, \unit{\alpha_1}$.
\end{itemize}
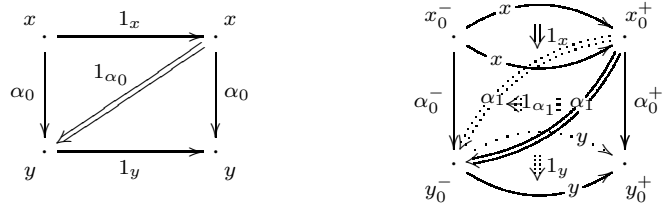
\[
  \begin{xy}
    \xymatrix{\overset{\makebox[0pt]{$\scriptstyle{x}\ \ \
          $}}{\cdot} 
         \ar[rr]^{\unit x}
         \ar[d]_{\alpha_0} &&
      \overset{\makebox[0pt]{$\ \ \ \scriptstyle{x}$}}{\cdot} 
      \ar[d]^{\alpha_0}
       \ar@{=>}[lld]_{\unit{\alpha_0}}\\
\underset{\makebox[0pt]{$\scriptstyle{y}\ \ \ $}}{\cdot}
        \ar[rr]_{\unit y}
& &\underset{\makebox[0pt]{$\ \ \ \scriptstyle{y}$}}{\cdot}  }
  \end{xy}
\hspace{5em}
  \begin{xy}
    \xymatrix{\ar@3{.>}(14,-9)*{};(7,-9)*{}|(0.4){\unit{\alpha_1}}
                      \ar@2{->}(11,1.5)*{};(11,-1.5)*{}^{\unit x}
                       \ar@2{.>}(11,-16)*{};(11,-19)*{}^{\unit y}
                     \overset{\makebox[0pt]{$\scriptstyle{\sce x_0}\ \ \ $}}{\cdot}
                       \ar@/^1pc/[rr]|(.3){\, x \,}
                      \ar@/_1pc/[rr]|(.25){\, x \,}
                      \ar[d]_{\sce \alpha_0}
      &&\overset{\makebox[0pt]{$\ \ \
          \scriptstyle{\tge x_0}$}}{\cdot}
                      \ar[d]^{\tge \alpha_0}
                       \ar@2{->}@/^1pc/[lld]|(.35){\alpha_1}
                       \ar[d]\ar@2{.>}@/_1pc/[lld]|(.67){\alpha_1}\\
           \underset{\makebox[0pt]{$\scriptstyle{\sce y_0}\ 
                            \ \ $}}{\cdot} 
                        \ar@{.>}@/^1pc/[rr]|(.75){\, y \,}
                         \ar@/_1pc/[rr]|(.7){\, y \,} 
      && \underset{\makebox[0pt]{$\ \ \ \scriptstyle{\tge y_0}$}}{\cdot}}
  \end{xy}
\]
\end{examples}

\begin{paragr}
Since this construction is clearly functorial, we get an endofunctor $\cyl : \ocat \to \ocat$, and by the above formulas,
the maps $\alpha \mapsto \tp \alpha$ and $\alpha \mapsto \bt\alpha$ define \oo-functors $\Top_C, \Bot_C : \Cyl C \to C$.
Hence, we also get two natural transformations $\Top, \Bot : \cyl \to \id \ocat$.
\end{paragr}

\begin{paragr}
Any $n$-cell $x$ yields a \ndef{trivial $n$-cylinder} $\Triv(x) : x \cto x$ given by the following~cells:
\begin{displaymath}
  \unit{\sce x_0}, \unit{\tge x_0}, \ldots, \unit{\sce x_{n-1}}, \unit{\tge x_{n-1}}, \unit{x_n}.
\end{displaymath}
It is in fact a unit for another composition of cylinders, which is called \ndef{concatenation} in \cite{LafMetPolRes}.
\end{paragr}

\begin{examples}\ %
\begin{itemize}
\item If $x$ is a 0-cell, the 0-cylinder $\Triv(x) : x \cto x$ is given by the cell $\unit x$.
\item If $x$ is a 1-cell, the 1-cylinder $\Triv(x) : x \cto x$ is given by the cells $\unit{\sce x_0}, \unit{\tge x_0}, \unit x$.
\item If $x$ is a 2-cell, the 2-cylinder $\Triv(x) : x \cto x$ is given by the cells $\unit{\sce x_0}, \unit{\tge x_0}, \unit{\sce x_1}, \unit{\tge x_1}, \unit x$.
\end{itemize}
\begin{displaymath}
  \begin{xy}
    \xymatrix{\overset{x}{\cdot} \ar[d]^{\unit x}\\
                       \underset{x}{\cdot}}
  \end{xy}
\hspace{5em}
  \begin{xy}
    \xymatrix{\overset{\makebox[0pt]{$\scriptstyle{\sce x_0}\ \ \
          $}}{\cdot} 
         \ar[rr]^{x}
         \ar[d]_{\unit{\sce x_0}} &&
      \overset{\makebox[0pt]{$\ \ \ \scriptstyle{\tge x_0}$}}{\cdot} 
      \ar[d]^{\unit{\tge x_0}}
       \ar@{=>}[lld]_{\unit x}\\
\underset{\makebox[0pt]{$\scriptstyle{\sce x_0}\ \ \ $}}{\cdot}
        \ar[rr]_x
& &\underset{\makebox[0pt]{$\ \ \ \scriptstyle{\tge x_0}$}}{\cdot}  }
  \end{xy}
\hspace{5em}
  \begin{xy}
    \xymatrix{\ar@3{.>}(14,-9)*{};(7,-9)*{}|(0.4){\unit x}
                      \ar@2{->}(11,1.5)*{};(11,-1.5)*{}^{x}
                       \ar@2{.>}(11,-16)*{};(11,-19)*{}^{x}
                     \overset{\makebox[0pt]{$\scriptstyle{\sce x_0}\ \ \ $}}{\cdot}
                       \ar@/^1pc/[rr]|(.3){\sce x_1}
                      \ar@/_1pc/[rr]|(.25){\tge x_1}
                      \ar[d]_{\unit{\sce x_0}}
      &&\overset{\makebox[0pt]{$\ \ \
          \scriptstyle{\tge x_0}$}}{\cdot}
                      \ar[d]^{\unit{\tge x_0}}
                       \ar@2{->}@/^1pc/[lld]|(.33){\, \unit{\tge x_1}}
                       \ar[d]\ar@2{.>}@/_1pc/[lld]|(.65){\unit{\sce x_1} \,\:}\\
           \underset{\makebox[0pt]{$\scriptstyle{\sce x_0}\ 
                            \ \ $}}{\cdot} 
                        \ar@{.>}@/^1pc/[rr]|(.75){\sce x_1}
                         \ar@/_1pc/[rr]|(.7){\tge x_1} 
      && \underset{\makebox[0pt]{$\ \ \ \scriptstyle{\tge x_0}$}}{\cdot}}
  \end{xy}
\end{displaymath}
\end{examples}

\goodbreak

\begin{definition}[oplax transformations] \label{def:homot} \ %

\noindent
If $f, g : C \to D$ are two \oo-functors, an \ndef{oplax transformation $\theta$ from $f$ to} $g$
is an \oo-functor $\theta : C \to \Cyl D$ making the following diagram commutative:
\[
    \xymatrix{
      & D \\
      C \ar[ur]^f \ar[dr]_g \ar[r]^-\theta & \Cyl{D} \ar[u]_{\Top_D} \ar[d]^{\Bot_D} \\
      & D}
\]
\end{definition}

\begin{paragr}\label{paragr:def_eq_homot}
More concretely, if $\theta$ is an oplax transformation from $f$ to $g$,
then we get an $n$-cylinder $\theta(x) : f(x) \cto g(x)$ in $D$ for each $n$-cell $x$ in $C$.
Its principal cell $\theta_x = \ppal{\theta(x)}$ is an $(n{+}1)$-cell~in~$D$, with the following source and target:
\[
  \theta_x : \theta_{\tge{x}_{n-1}} \!\! \comp{n-1} \cdots \comp 1 \theta_{\tge{x}_0} \comp 0 f(x)
    \to g(x) \comp 0 \theta_{\sce{x}_0} \comp 1 \cdots \comp{n-1} \theta_{\sce{x}_{n-1}}.
\]
By \oo-functoriality of $\theta$ and by construction of the
\oo-category~$\Cyl D$, the following two equalities hold for $n > p$, for
any $n$-cells $x, y$ such that $\tge x_p = \sce y_p$, and for any cell $u$:
\[
\begin{array}{c}
\theta_{y \comp{p} x} = \left( g(\tge{y}_{p+1}) \comp 0 \theta_{\sce{x}_0} \comp 1 \cdots \comp{p-1} \theta_{\sce{x}_{p-1}} \!\! \comp p \theta_x \right)
\comp{p+1} \left( \theta_y \comp p \theta_{\tge{y}_{p-1}} \!\! \comp{p-1} \cdots \comp 1 \theta_{\tge{y}_0} \comp 0 f(\sce{x}_{p+1}) \right), \myspace \\
\theta_{\unit u} = \unit{\theta_u}.
\end{array}
\]
Conversely, if for each $n$-cell $x$ in $C$, $\theta_x$ is an $(n{+}1)$-cell in $D$ with the above source and target, and if the above two axioms hold,
then we get a unique oplax transformation $\theta$ from $f$ to $g$, which is defined as follows for each $n$-cell $x$ in $C$:
\[
\eps{\theta(x)}_i = \theta_{\eps{x}_i} \mbox{ for } i < n, \qquad
\ppal{\theta(x)} = \theta(x)_n = \theta_x.
\]
In other words, $\theta$ can be reconstructed from the $\theta_x$.
See \cite[Section B.2]{AraMaltsiJoint} for more details.
\end{paragr}

\begin{example}
For any \oo-category $C$, the \oo-functor $\Triv : C \to \Cyl C$ mapping any cell $x$ in $C$ to the trivial cylinder $\Triv(x) : x \cto x$
defines an oplax transformation from $\id C$ to itself, which is given by the following $(n{+}1)$-cell for each $n$-cell $x$ in $C$:
\[
\Triv_x = \unit x : x \to x.
\]
Hence, we get a natural transformation $\Triv : \id \ocat \to \cyl$, which is a common section of $\Top$ and $\Bot$. 
\end{example}

\section{Orientals from the expansion monad}\label{monad}

This section addresses the main goal of this work, namely a construction of the \ndef{cosimplicial object of orientals} $\orient : \scat \to \ocat$,
which is obtained by iterating a monad $\exmon : \ocat \to \ocat$. 
In particular, we get the following definition of orientals:
\[
\norient 0 = \exmon(\emptyset), \norient 1 = \exmon^2(\emptyset), \norient 2 = \exmon^3(\emptyset), \ldots
\]
This monad comes from the forgetful functor from the category of \oo-categories with expansion to the category of \oo-categories.
This notion of \ndef{expansion} is central in this paper:
\begin{itemize}
\item It was first introduced under the name of \oo-\ndef{initial structure} by Burroni in the unpublished paper \cite{BurOrientPreprint}.
\item It was then introduced independently under the name of \ndef{contraction} by the first author and Maltsiniotis in \cite{AraMaltsiCondE},
in their study of homotopical properties of orientals.
\end{itemize}

\subsection{Cones}

\begin{definition}[\oo-category of small cones] \label{def:cone} \ %

\noindent
If $\orig$ is a 0-cell in an \oo-category $C$, which amounts to an
\oo-functor $\orig: \Sing \to C$ where $\Sing$ is the terminal
\oo-category, the \oo-\ndef{category $\Cone C$ of small cones of origin
$\orig$ in} $C$ is given by the following pullback square:
\[
  \begin{xy}
    \xymatrix{\Cone C \, \ar[r]\ar[d]\ar@{}@<-3pt>[rd]|(0.2){\displaystyle{\lrcorner}}& \Cyl C \ar[d]^{\Top_C}\\
                      \Sing \, \ar[r]_{\orig} & C }
  \end{xy}
\]
\end{definition}

\begin{remark}
As the bottom arrow is a monomorphism, so is the top one, and in fact, $\Cone C$ is a full subcategory of $\Cyl C$.
\end{remark}

\begin{paragr}\label{paragr:cone_formulas}
More concretely, an $n$-cell in $\Cone C$ amounts to an $n$-cylinder $\alpha : \Unit n \orig \cto x$ in $C$,
which~is written $\alpha : \orig \cto x$ and called $n$-\ndef{cone of origin} $\orig$.
The $n$-cell $\bt \alpha = x$ is called the \ndef{basis} of the~cone~$\alpha$.

The formulas of paragraph~\ref{paragr:conc} for sources of (auxiliary and principal) cells of cylinders are simpler in the case of cones.
Indeed, an $n$-cone $\alpha: \orig \cto x$ is given by a finite sequence of cells
\[
\sce \alpha_0, \tge \alpha_0, \ldots, \sce \alpha_{n-1}, \tge \alpha_{n-1}, \alpha_n
\]
with the following sources and targets:
\[
\begin{array}{c}
\eps \alpha_0 : \orig \to \eps x_0, \qquad
\eps \alpha_i : \tge \alpha_{i-1} \to \eps x_i \! \comp 0 \sce \alpha_0 \! \comp 1 \cdots \comp {i-1} \sce \alpha_{i-1}
\mbox{ for } 0 < i < n, \myspace \\
\ppal \alpha = \alpha_0 : \orig \to x \mbox{ if } n = 0, \qquad
\ppal \alpha = \alpha_n : \tge \alpha_{n-1} \to x \comp 0 \sce \alpha_0 \! \comp 1 \cdots \comp {n-1} \sce \alpha_{n-1} \mbox{ if } n > 0.
\end{array}
\]
Note that the formulas for targets are unchanged, except that $y$ is replaced by $x$.
\end{paragr}

\begin{examples}\ %
\begin{itemize}
\item If $x$ is a 0-cell, a 0-cone $\alpha : \orig \cto x$ is given by the 1-cell $\ppal \alpha = \alpha_0 : \orig \to x$.
\item If $x$ is a 1-cell, a 1-cone $\alpha : \orig \cto x$ is given by the 1-cells $\eps \alpha_0 : \orig \to \eps x_0$
and the 2-cell $\ppal \alpha = \alpha_1 : \tge \alpha_0 \to x \comp 0 \sce \alpha_0$.
\item If $x$ is a 2-cell, a 2-cone $\alpha : \orig \cto x$ is given by the 1-cells $\eps \alpha_0 : \orig \to \eps x_0$,
the 2-cells $\eps \alpha_1 : \tge \alpha_0 \to \eps x_1 \comp 0 \sce \alpha_0$,
and the 3-cell $\ppal \alpha = \alpha_2 : \tge \alpha_1 \to x \comp 0 \sce \alpha_0 \comp 1 \sce \alpha_1$.
\end{itemize}
\[
  \begin{xy}
   \xymatrix{\overset{\orig}{\cdot} \ar[d]^{\alpha_0}\\
                       \underset{x}{\cdot}}
  \end{xy}
\hspace{5em}
  \begin{xy}
    \xymatrix @=1.5em{
&\overset{\orig}{\cdot} 
           \ar[ldd]_{\sce \alpha_0} 
            \ar[rdd]^{\tge \alpha_0}
         &
        \ar@2{->}(11,-13)*{};(7,-16)*{}_(.3){\alpha_1}\\
&&\\
\underset{\makebox[0pt]{$\scriptstyle{\sce x_0}\ \ \ $}}{\cdot} 
\ar[rr]_x
&&
\underset{\makebox[0pt]{$\ \ \ \scriptstyle{\tge x_0}$}}{\cdot}  }
  \end{xy}
\hspace{5em}
  \begin{xy}
    \xymatrix @=1.2em{
&&
\overset{\orig}{\cdot} 
          \ar[lldddd]_{\sce \alpha_0} 
           \ar[rrdddd]^{\tge \alpha_0}
            \ar@2{.>}(15,-30)*{};(15,-34)*{}^(.3){x}
&&\\
&&&& \\
&&&         
             \ar@2@/^1pc/[ddlll]|(.27){\tge \alpha_1}
              \ar@2{.>}@/_.8pc/[ddlll]|(.6){\sce \alpha_1}
              \ar@3{.>}(15,-23)*{};(9,-23)*{}_(.4){\alpha_2}&
\\
&&&&
\\
     \underset{\makebox[0pt]{$\scriptstyle{\sce x_0}\ \ \ $}}{\cdot}
              \ar@{.>}@/^1pc/[rrrr]|(.6){\sce x_1}
               \ar@/_1pc/[rrrr]|(.6){\tge x_1}
&&&&
      \underset{\makebox[0pt]{$\ \ \ \scriptstyle{\tge x_0}$}}{\cdot}  }
  \end{xy}
\]
\end{examples}

\begin{paragr}
The formulas of definition~\ref{def:comp} for $p$-composition are simpler in the case of cones.
Indeed, if $n > p$, if $x, y$ are $n$-cells such that $\tge x_p = \sce y_p$,
and if $\alpha : \orig \cto x, \beta : \orig \cto y$ are $n$-cones such that $\TGE_p \alpha = \SCE_p \beta$,
the $p$-composition $\gamma = \beta \comp p \alpha: \orig \cto y \comp p x$ is the $n$-cone given by the following cells:
\[ 
\begin{array}{c}
\eps \gamma_i = \eps \alpha_i = \eps \beta_i \mbox { for } i < p, \qquad
\sce \gamma_p = \sce \alpha_p, \qquad
\tge \gamma_p = \tge \beta_p, \myspace \\
\eps \gamma_{p+1} = \eps y_{p+1} \comp 0 \sce \alpha_0 \comp 1 \cdots \comp{p-1} \sce \alpha_{p-1} \comp p \eps \alpha_i \comp{p+1} \eps \beta_{p+1}
\mbox { if } p+1 < n, \myspace \\
\eps \gamma_i = \tge y_{p+1} \comp 0 \sce \alpha_0 \comp 1 \cdots \comp{p-1} \sce \alpha_{p-1} \comp p \eps \alpha_i \comp{p+1} \eps \beta_i
\mbox { for } p+1 < i < n, \myspace \\
\ppal \gamma = \gamma_n = \tge y_{p+1} \comp 0 \sce \alpha_0 \comp 1 \cdots \comp{p-1} \sce \alpha_{p-1} \comp p \alpha_n \comp{p+1} \beta_n.
\end{array}
\]
On the other hand, the formulas for $p$-units are unchanged.
\end{paragr}

\goodbreak

\begin{examples}\ %
\begin{itemize}
\item If $x, y$ are 1-cells such that $\tge x_0 = \sce y_0$, and if $\alpha : \orig \cto x, \beta : \orig \cto y$ are 1-cones
such that $\tge \alpha_0 = \sce \beta_0$, the 0-composition $\gamma = \beta \comp 0 \alpha : \orig \cto y \comp 0 x$
is given by the 1-cells $\sce \gamma_0 = \sce \alpha_0$ and $\tge \gamma_0 = \tge \beta_0$,
and the 2-cell $\gamma_1 = y \comp 0 \alpha_1 \comp 1 \beta_1$.
\item If $x, y$ are 2-cells such that $\tge x_0 = \sce y_0$, and if $\alpha : \orig \cto x, \beta : \orig \cto y$ are 2-cones
such that $\tge \alpha_0 = \sce \beta_0$, the 0-composition $\gamma = \beta \comp 0 \alpha : \orig \cto y \comp 0 x$
is given by the 1-cells $\sce \gamma_0 = \sce \alpha_0$ and $\tge \gamma_0 = \tge \beta_0$,
the 2-cells $\eps \gamma_1 = \eps y_1 \comp 0 \eps \alpha_1 \comp 1 \eps \beta_1$,
and the 3-cell $\gamma_2 = \tge y_1 \comp 0 \alpha_2 \comp 1 \beta_2$.
\item If $x, y$ are 2-cells such that $\tge x_1 = \sce y_1$, and if $\alpha : \orig \cto x, \beta : \orig \cto y$ are 2-cones
such that $\eps \alpha_0 = \eps \beta_0$ and $\tge \alpha_1 = \sce \beta_1$,
the 1-composition $\gamma = \beta \comp 1 \alpha : \orig \cto y \comp 1 x$
is given by the 1-cells $\eps \gamma_0 = \eps \alpha_0 = \eps \beta_0$,
the 2-cells $\sce \gamma_1 = \sce \alpha_1$ and $\tge \gamma_1 = \tge \beta_1$,
and the 3-cell $\gamma_2 = y \comp 0 \sce \alpha_0 \comp 1 \alpha_2 \comp 2 \beta_2$.
\end{itemize}
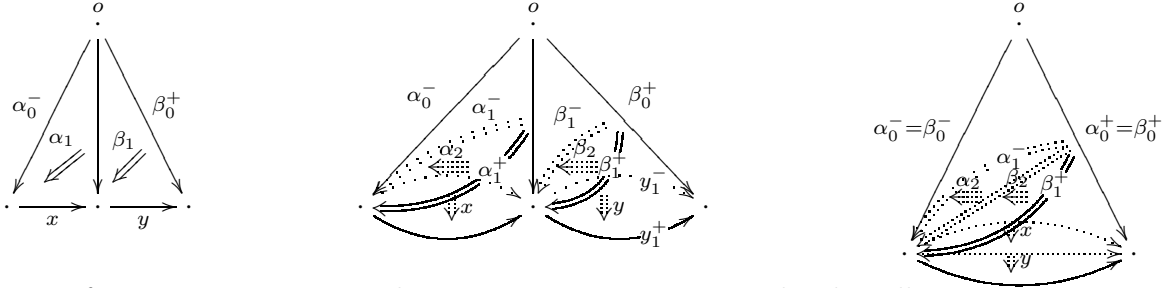
\[
  \begin{xy}
    \xymatrix{
&\overset{\orig}{\cdot} 
        \ar[ldd]_{\sce \alpha_0} 
        \ar[dd] 
        \ar[rdd]^{\tge \beta_0}
       &
      \ar@2{->}(10,-17)*{};(5,-21)*{}_(.1){\alpha_1}
       \ar@2{->}(18,-17)*{};(14,-21)*{}_(.1)*-<3pt>{\scriptstyle{\beta_1}}\\
&&\\
\cdot 
\ar[r]_x
& \cdot
\ar[r]_y
&\cdot}
  \end{xy}
\hspace{5em}
  \begin{xy}
    \xymatrix{
&&
     \overset{\orig}{\cdot} 
          \ar[dd]
           \ar[lldd]_{\sce \alpha_0} 
           \ar[rrdd]^{\tge \beta_0}
            \ar@2{.>}(12,-23)*{};(12,-26)*{}^{x}
            \ar@2{.>}(32,-23)*{};(32,-26)*{}^{y}
&&\\
& &         \ar@2@/^{1pc}/[dll]|(.3){\tge \alpha_1}
              \ar@2{.>}@/_/[dll]_(.2){\sce \alpha_1}
              \ar@3{.>}(14,-19)*{};(9,-19)*{}_(.4){\alpha_2}
               \ar@3{.>}(31,-19)*{};(26,-19)*{}_(.3){\beta_2}
&
\ar@2@/^{1pc}/[dl]|(.3){\tge \beta_1}
              \ar@2{.>}@/_/[dl]_(.3){\sce \beta_1}
& \\
           \cdot
              \ar@{.>}@/^1pc/[rr]
               \ar@/_1pc/[rr]
&&
      \cdot  
\ar@{.>}@/^1pc/[rr]|(.7){\sce y_1}
               \ar@/_1pc/[rr]|(.7){\tge y_1}
&& \cdot }
  \end{xy}
\hspace{5em}
  \begin{xy}
    \xymatrix @=1.2em{
&&
\overset{\orig}{\cdot} 
          \ar[lldddd]_{\sce \alpha_0=\sce \beta_0}
           \ar[rrdddd]^{\tge \alpha_0=\tge \beta_0}
             \ar@2{.>}(14,-27)*{};(14,-29)*{}^(.3){x}
            \ar@2{.>}(14,-31)*{};(14,-33)*{}^(.3){y}
&&\\
&&&& \\
&&&         
             \ar@2@/^{1.1pc}/[ddlll]|(.2){\tge \beta_1}
              \ar@2{.>}@/_{.9pc}/[ddlll]|(.3){\sce \alpha_1}
               \ar@2{.>}[ddlll]
\ar@3{.>}(16,-23)*{};(13,-23)*{}_(.4){\beta_2}
\ar@3{.>}(10,-23)*{};(6,-23)*{}_(.4){\alpha_2}              
&
\\
&&&&
\\
     \cdot
              \ar@{.>}@/^{1pc}/[rrrr]
               \ar@{.>}[rrrr]
               \ar@/_{1pc}/[rrrr]
&&&&
     \cdot  }
  \end{xy}
\]
\begin{itemize}
\item If $\alpha : \orig \cto x$ is a 0-cone, the 0-unit $\unit \alpha : \orig \cto \unit x$ is given by the cells $\alpha_0, \alpha_0, \unit{\alpha_0}$.
\item If $\alpha : \orig \cto x$ is a 1-cone, the 1-unit $\unit \alpha : \orig \cto \unit x$ is given by the cells
$\sce \alpha_0, \tge \alpha_0, \alpha_1, \alpha_1, \unit{\alpha_1}$.
\end{itemize}
\[
  \begin{xy}
    \xymatrix @=1.5em{
&\overset{\orig}{\cdot} 
           \ar[ldd]_{\alpha_0} 
            \ar[rdd]^{\alpha_0}
         &
        \ar@2{->}(11,-13)*{};(7,-16)*{}_(.3){\unit{\alpha_0}}\\
&&\\
\underset{\makebox[0pt]{$\scriptstyle{x}\ \ \ $}}{\cdot} 
\ar[rr]_{\unit x}
&&
\underset{\makebox[0pt]{$\ \ \ \scriptstyle{x}$}}{\cdot}  }
  \end{xy}
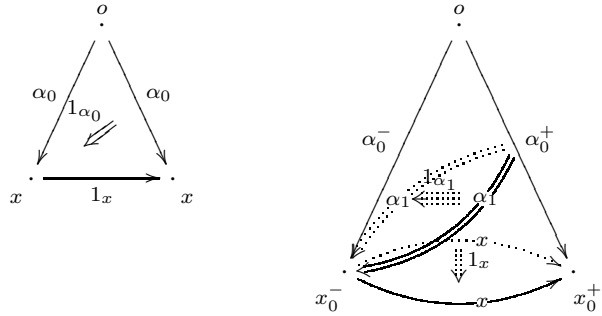
\hspace{5em}
  \begin{xy}
    \xymatrix @=1.2em{
&&
\overset{\orig}{\cdot} 
          \ar[lldddd]_{\sce \alpha_0} 
           \ar[rrdddd]^{\tge \alpha_0}
            \ar@2{.>}(15,-30)*{};(15,-34)*{}^(.3){\unit x}
&&\\
&&&& \\
&&&         
             \ar@2@/^1pc/[ddlll]|(.27){\alpha_1}
              \ar@2{.>}@/_.8pc/[ddlll]|(.6){\alpha_1}
              \ar@3{.>}(15,-23)*{};(9,-23)*{}_(.4){\unit{\alpha_1}}&
\\
&&&&
\\
     \underset{\makebox[0pt]{$\scriptstyle{\sce x_0}\ \ \ $}}{\cdot}
              \ar@{.>}@/^1pc/[rrrr]|(.6){x}
               \ar@/_1pc/[rrrr]|(.6){x}
&&&&
      \underset{\makebox[0pt]{$\ \ \ \scriptstyle{\tge x_0}$}}{\cdot}  }
  \end{xy}
\]
\end{examples}

\begin{remark}
The only trivial $n$-cylinder defining an $n$-cone of origin $\orig$ is $\Triv(\Unit n \orig) : \Unit n \orig \cto \Unit n \orig$,
but there is a weaker notion, which is used to define the notion of expansion:
\end{remark}

\begin{definition}[degenerate cones]\
  
 \noindent 
An $n$-cone $\alpha : \orig \cto x$ is called \ndef{degenerate} if its principal cell $\ppal \alpha = \alpha_n$ is a unit,
as well as its \ndef{negative auxiliary cells} $\sce \alpha_0, \ldots, \sce \alpha_{n-1}$.
\end{definition}

\begin{remark}
  If $\alpha : \orig \cto x$ is a degenerate $n$-cone, then $\sce
  x_0=\orig$, as $\sce\alpha_0$ is an identity.
\end{remark}

\begin{lemma}\label{lemma:unique_deg_cone}
  Let $x$ be an $n$-cell such that $\sce x_0=\orig$. There is a unique
  degenerate cone $\alpha : \orig \cto x$ of base $x$. This cone is
  determined by the following cells:
  \[
    \begin{aligned}
      \sce\alpha_i & = \unit{\sce x_i} \quad\text{for $0 \le i \le n - 1$,}
    \\
      \tge\alpha_i & =
      \begin{cases}
        \sce x_{i+1} & \text{for $0 \le i \le n - 2$,} \\
        x & \text{for $i = n-1$,}
      \end{cases}
      \\
      \alpha_n & = \unit{x}.
    \end{aligned}
  \]
\end{lemma}
\begin{proof}
  By induction  on $n$.
  \begin{itemize}
  \item For $n=0$, there is a unique degenerate cone $\alpha:\orig\cto
    x$ determined by $x=\orig$ and $\alpha_0=\unit\orig$.
  \item Suppose that the statement holds up to dimension $n$ and let us show
    that the it holds in dimension $n+1$. Thus, let $x$ be an
    $(n{+}1)$-cell such that $\sce x_0=\orig$. For the uniqueness
    part, suppose that
    $\alpha:\orig\cto x$ is a degenerate $(n{+}1)$-cone and consider the
    $n$-cone $\beta=\SCE\alpha:\orig\cto\sce
    x_n$. By~\ref{paragr:conc}, $\sce\beta_i=\sce\alpha_i$ for
    $0 \le i \le n-1$ and $\beta_n=\sce\alpha_n$, so that all
    negative auxiliary cells of $\beta$ are identities, as well as its
    principal cell $\beta_n$. By induction hypothesis, $\beta$ is the
    unique degenerate $n$-cone of base $\sce x_n$, and we know
    by~\ref{paragr:conc} that $\eps\alpha_i=\eps\beta_i$ for
    $0 \le i \le n-1$ and $\epsilon=\pm$. Now by~\ref{paragr:cone_formulas}
    the principal cell of $\alpha$ is
    \[\alpha_{n+1} : \tge\alpha_n\to x\comp 0\sce\alpha_0\comp
    1\cdots\comp{n}\sce\alpha_n\]
    and because $\alpha$ is degenerate, the right-hand side of the
    previous formula is just $x$, whereas $\alpha_{n+1}$ is an
    identity so that $\tge\alpha_n=x$ and $\alpha_{n+1}=\unit x$. This
    proves uniqueness. Finally, the above formulas for $\eps\alpha_i$,
    $0\leq i\leq n$, and $\alpha_{n+1}$ satisfy the relations
    of~\ref{paragr:cone_formulas} and define a degenerate $(n{+}1)$-cone.
    \qedhere
  \end{itemize}
\end{proof}

\begin{examples}
For $n = 0, 1, 2$, we get the following degenerate cones:
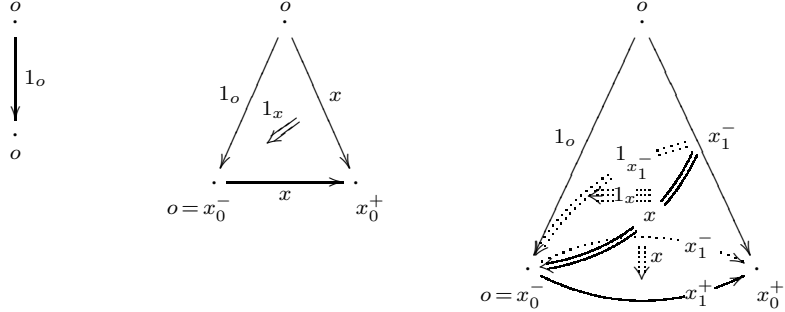
\[
  \begin{xy}
   \xymatrix{\overset{\orig}{\cdot} \ar[d]^{\unit \orig}\\
                       \underset{\orig}{\cdot}}
  \end{xy}
\hspace{5em}
  \begin{xy}
    \xymatrix @=1.5em{
&\overset{\orig}{\cdot}
           \ar[ldd]_{\unit{\orig}} 
            \ar[rdd]^{x}
         &
        \ar@2{->}(11,-13)*{};(7,-16)*{}_(.3){\unit x}\\
&&\\
\underset{\makebox[0pt]{$\scriptstyle{\orig \, = \, \sce x_0}\ \ \ $}}{\cdot} 
\ar[rr]_x
&&
\underset{\makebox[0pt]{$\ \ \ \scriptstyle{\tge x_0}$}}{\cdot}  }
  \end{xy}
\hspace{5em}
  \begin{xy}
    \xymatrix @=1.2em{
&&
\overset{\orig}{\cdot} 
          \ar[lldddd]_{\unit{\orig}} 
           \ar[rrdddd]^{\sce x_1}
            \ar@2{.>}(15,-30)*{};(15,-34)*{}^(.3){x}
&&\\
&&&& \\
&&&         
             \ar@2@/^1pc/[ddlll]|(.4)*+{\scriptstyle{x}}
              \ar@2{.>}@/_.8pc/[ddlll]|(.3){\unit{\sce x_1}}
              \ar@3{.>}(16,-23)*{};(8,-23)*{}|(.4){\unit x}&
\\
&&&&
\\
     \underset{\makebox[0pt]{$\scriptstyle{\orig \, = \, \sce x_0}\ \ \ $}}{\cdot}
              \ar@{.>}@/^1pc/[rrrr]|(.75){\sce x_1}
               \ar@/_1pc/[rrrr]|(.75){\tge x_1}
&&&&
      \underset{\makebox[0pt]{$\ \ \ \scriptstyle{\tge x_0}$}}{\cdot}  }
  \end{xy}
\]
\end{examples}

\begin{paragr}
We write $\ocatpt$ for the \ndef{category of pointed \oo-categories}:
\begin{itemize}
\item An object is a pair $(C, \orig)$ where $C$ is an \oo-category,
and $\orig$ is a 0-cell in $C$, called \ndef{origin}.
\item A morphism $f : (C, \orig) \to (D, \orig)$ is an \oo-functor $f : C \to D$ preserving the origin.
\end{itemize}
\end{paragr}

\begin{paragr}
For any such morphism, we get two \oo-functors $f : C \to D$ and $\Cyl f : \Cyl  C \to \Cyl D$,
which induce an \oo-functor $\cone(f) : \Cone C \to \Cone D$ by the pullback square of definition~\ref{def:cone}.
Hence, we get a functor $\cone : \ocatpt \to \ocat$.

The natural transformation $\Bot : \cyl \to \id \ocat$ induces a natural transformation $\Bas : \cone \to \proj$,
where $\proj : \ocatpt \to \ocat$ stands for the forgetful functor.
This means that, for any pointed \oo-category $(C, \orig)$, we get an \oo-functor $\Bas_{(C, \orig)} : \Cone C \to C$, which maps any cone to its basis.
In practice, we shall simply write $\Bas_C$ for this \oo-functor.
\end{paragr}

\begin{paragr} \label{paragr:homot_cone}
The \oo-category of small cones can be used to describe particular oplax transformations.
Indeed, by the pullback square of definition~\ref{def:cone}, if $C, D$ are \oo-categories and $\orig$ is a 0-cell in $D$,
an oplax transformation from the constant \oo-functor $\orig : C \to D$ to another \oo-functor $f : C \to D$
amounts to an \oo-functor $\theta : C \to \Cone D$ making the following triangle commutative:
\[
    \xymatrix{
      C \ar[r]^-\theta \ar[dr]_f & \Cone{D} \ar[d]^{\Bas_D} \\
      & D
    }
\]
In particular, an oplax transformation from the constant \oo-functor $\orig : C \to C$ to the identity \oo-functor $\id C : C \to C$
amounts to a section $\theta : C \to \Cone C$ of $\Bas_C : \Cone C \to C$.
\end{paragr}

\begin{paragr} \label{paragr:def_eq_homot_orig}
By the formulas of paragraph~\ref{paragr:def_eq_homot}, such an oplax transformation amounts to
the data of an $(n{+}1)$-cell $\theta_x$ for each $n$-cell $x$ in $C$, with the following source and target:
\[
\theta_x : \orig \to x \mbox{ for } n = 0, \qquad
\theta_x : \theta_{\tge{x}_{n-1}} \!\! \to x \comp 0 \theta_{\sce{x}_0} \comp 1 \cdots \comp{n-1} \theta_{\sce{x}_{n-1}} \mbox{ for } n > 0,
\]
such that the following two axioms hold for $n > p$, for any $n$-cells $x, y$ such that $\tge x_p = \sce y_p$, and for any cell $u$:
\[
\theta_{y \comp p x} = \tge{y}_{p+1} \comp 0 \theta_{\sce{x}_0} \comp 1 \cdots \comp{p-1} \theta_{\sce{x}_{p-1}} \!\! \comp p \theta_x \comp{p+1} \theta_y, \qquad
\theta_{\unit u} = \unit{\theta_u}.
\]
\end{paragr}

\subsection{Expansion monad}

\begin{definition}[expansions] \ %

\noindent
If $C$ is an \oo-category, an \ndef{expansion} on $C$ consists of:
\begin{itemize}
\item a $0$-cell $\orig$ in $C$, called \ndef{origin},
\item a section $\econe : C \to \Cone C$ of the \oo-functor $\Bas_C : \Cone C \to C$, called \ndef{expanding homotopy},
\end{itemize}
such that the cone $\econe(x)$ is degenerate whenever $x$ is of the form $\orig$ or $\ppal{\econe(u)}$ for some $u$.
\end{definition}

\begin{paragr} \label{paragr:exp}
By paragraph~\ref{paragr:homot_cone}, $\econe$ is an oplax transformation from the constant \oo-functor $\orig : C \to C$
to the identity \oo-functor $\unit C : C \to C$.
More concretely, it amounts by paragraph~\ref{paragr:def_eq_homot_orig} to the data of
an $(n{+}1)$-cell $\econe_x$ for each $n$-cell $x$ in $C$, with the following source and target:
\[
\econe_x : \orig \to x \mbox{ for } n = 0, \qquad
\econe_x : \econe_{\tge{x}_{n-1}} \!\! \to x \comp 0 \econe_{\sce{x}_0} \comp 1 \cdots \comp{n-1} \econe_{\sce{x}_{n-1}} \mbox{ for } n > 0,
\]
such that the following four axioms hold for $n > p$, for any $n$-cells $x, y$ such that $\tge x_p = \sce y_p$, and for any cell $u$:
\[
\econe_{y \comp p x} =
\tge{y}_{p+1} \comp 0 \econe_{\sce{x}_0} \comp 1 \cdots \comp{p-1} \econe_{\sce{x}_{p-1}} \!\! \comp p \econe_x \comp{p+1} \econe_y, \qquad
\econe_{\unit u} = \unit{\econe_u}, \qquad
\econe_{\econe_u} = \unit{\econe_u}, \qquad
\econe_\orig = \unit \orig.
\]
We will sometimes call the first two axioms the \ndef{functoriality
conditions} and the last two axioms the \ndef{degeneracy conditions}.
\end{paragr}

\begin{paragr}
We write $\expcat$ for the \ndef{category of \oo-categories with expansion}:
\begin{itemize}
\item An object is a triple $(C, \orig, \econe)$, where $C$ is an \oo-category, and $\orig, \econe$ define an expansion on~$C$.
\item A morphism $f : (C, \orig, \econe) \to (D, \orig, \econe)$ is an \oo-functor $f : C \to D$ preserving the structure,
which means that $f(\orig) = \orig$ and the following square commutes:
\[
  \xymatrix{
     C \ar[r]^-\econe \ar[d]_f & \Cone C \ar[d]^{\cone(f)} \\
     D \ar[r]_-{\econe} & \Cone D \\
  }
 \]
\end{itemize}
\end{paragr}

\begin{proposition} \label{prop:free_exp}
The obvious forgetful functor $\fgf : \expcat \to \ocat$ admits a left adjoint.
\end{proposition}

\begin{proof}
This follows from the fact that our structures are ``equational'' in the sense of the theory of sketches.
See for instance \cite{AdamRos} for an introduction to this theory.
See also Remark~\ref{rem:exist_adj} for a more concrete proof of the existence of this left adjoint.

Following the usual definition of \oo-categories, $\ocat$ is indeed sketchable by a limit sketch~$\sketch$:
\begin{itemize}
\item Objects are the following symbols:
\[
\begin{array}{c}
C_n \mbox{ for } n \ge 0, \qquad
C_n \times_{C_p} C_n \mbox{ and } C_n \times_{C_p} C_n \times_{C_p} C_n \mbox{ for } n > p \ge 0, \myspace \\
(C_n \times_{C_p} C_n) \times_{C_q} (C_n \times_{C_p} C_n) \mbox{ for } n > p > q \ge 0.
\end{array}
\]
\item Generators (for morphisms) are given by sources and targets, compositions and units.
\item Relations are given by laws of associativity, unit, interchange and functoriality of units.
\item Distinguished cones are suggested by the notation of objects.
\end{itemize}
Similarly, the equational definition of paragraph~\ref{paragr:exp}, produces a limit sketch $\esketch$ whose models~are \oo-categories with expansion.
More precisely, $\esketch$ is obtained by adding to $\sketch$ the object $\sing$ as~well~as suitable generators, relations and distinguished cones.

Now, the canonical inclusion of $\sketch$ into $\esketch$ defines a morphism of sketches $\iota : \sketch \emb \esketch$,
and the induced functor $\Mod \iota : \Mod \esketch \to \Mod \sketch$ is $\fgf : \expcat \to \ocat$.
By \cite[Lemma~p.~6]{LairMonad}, this implies that $\fgf$ admits a left adjoint.
\end{proof}

\begin{paragraph}
We write $\frf : \ocat \to \expcat$ for this left adjoint, $\exmon = \fgf \frf : \ocat \to \ocat$ for the induced \ndef{expansion monad},
$\mult : T^2 \to T$ for its multiplication, and $\munit : \id \ocat \to T$ for its unit.
\end{paragraph}

It happens that the algebras of this monad are precisely the \oo-categories with expansion:

\begin{proposition}
The forgetful functor $\fgf : \expcat \to \ocat$ is monadic.
\end{proposition}

\begin{proof}
We use again the sketches $\sketch, \esketch$ and the morphism $\iota : \sketch \emb \esketch$ introduced in the proof of Proposition~\ref{prop:free_exp}.
This morphism has the following properties:
\begin{itemize}
\item The base of any distinguished cone of $\esketch$ factors though $\iota$.
\item Every object of $\esketch$ not reached by $\sketch$ (namely only $\sing$) is the tip of a distinguished cone.
\end{itemize}
It thus fulfills the hypothesis of \cite[Corollary 1]{LairMonad} and it follows that $\fgf$ is monadic.
\end{proof}

\subsection{Cosimplicial object of orientals}

\begin{paragraph}
We write $\scat$ for the \ndef{simplex category}:
\begin{itemize}
\item Its objects are the ordered sets $\deltan{n} = \set{0 < 1 < \cdots < n}$ for $n \ge 0$.
\item Its morphisms are the order-preserving maps.
\end{itemize}
The morphisms of $\scat$ are generated by
\[
\faced i n : \deltan{n-1} \to \deltan{n}, \mbox{ for } n > 0 \mbox{ and } 0 \le i \le n, \qquad
\degend i n : \deltan{n+1} \to \deltan{n}, \mbox{ for } n \geq 0 \mbox{ and } 0 \le i \le n,
\]
where $\faced i n$ is the unique order-preserving injection such that the preimage of $\set i$ is empty,
and $\degend i n$ is the unique order-preserving surjection such that the preimage of $\set i$ has two elements.
\end{paragraph}

\begin{paragraph}
Similarly, we write $\scata$ for the \ndef{augmented simplex category}:
\begin{itemize}
\item Its objects are those of $\scat$, plus an additional one: $\deltan{-1} = \emptyset$.
\item Its morphisms are again the order-preserving maps.
\end{itemize}
By definition, $\scat$ is a full subcategory of $\scata$.
Moreover, the morphisms of $\scata$ are generated by the generating
morphisms of $\scat$, plus an additional one: $\faced{0}{0} : \deltan{-1} \to
\deltan{0}$.
\end{paragraph}

\begin{paragraph}
Recall that $\scata$ is the universal monoidal category endowed with a monoid object.
More precisely, the disjoint union $\deltan{m} \amalg \deltan{n} = \deltan{m + 1 + n}$ defines a strict monoidal structure on $\scata$,
with unit $\emptyset = \deltan{-1}$, and $\deltan{0}$ is endowed with a unique structure of monoid for this monoidal structure:
\[
\degend 0 0 : \deltan{0} \amalg \deltan{0} =  \deltan{1} \to \deltan{0}, \qquad
\faced 0 0 : \emptyset = \deltan{-1} \to \deltan{0}.
\]
The universal property of $\scata$ can then be expressed as follows (see for
instance \cite[Chapter~VII, Section~5]{MacLane}):
\end{paragraph}

\begin{lemma}
For any monoid object $M$ in a strict monoidal category $\mathcal C$,
there exists a unique strict monoidal functor $\Phi : \scata \to \mathcal C$
sending the monoid $\deltan{0}$ to the monoid $M$.
\end{lemma}

\begin{paragraph}
In particular, a monad on a category $\mathcal C$ amounts to a monoid object
in the strict monoidal category $\End(\mathcal C)$ of endofunctors on $\mathcal C$.
Hence, for any such monad $T$, we get a canonical functor $c^{}_T : \scata \to \End(\mathcal C)$,
which is given as follows on objects and on generators:
\[
c^{}_T(\deltan{n}) = T^{n+1}, \qquad
c^{}_T(\faced i n) = T^{n-i} \munit T^i : T^n \to T^{n+1}, \qquad
c^{}_T(\degend i n) = T^{n-i} \mult T^i : T^{n+2} \to T^{n+1}.
\]
\end{paragraph}

\begin{definition}[orientals] \label{def:orientals} \ %

\noindent
The \ndef{augmented cosimplicial object of orientals} $\orienta : \scata \to \ocat$ is the composition
\[
\scata \xrightarrow{c^{}_\exmon} \End(\ocat) \xrightarrow{\ev{\emptyset}} \ocat
\]
where $c^{}_\exmon$ is given in the previous paragraph and $\ev \emptyset$ is the evaluation functor at $\emptyset$.

By restricting $\orienta$ to $\scat$, we get the \ndef{cosimplicial object of orientals} $\orient : \scat \to \ocat$,
and for $n \ge 0$, the $n$-\ndef{th oriental} is $\norient n =
\orient(\deltan{n})$.
\end{definition}

\begin{paragraph} \label{paragr:free-monad}
Explicitly, we have:
\[
\norient{n} = \exmon^{n+1}(\emptyset), \qquad
\orient(\faced{i}{n}) = \exmon^{n-i} \munit \exmon^{i}(\emptyset) : \norient{n-1} \to \norient{n}, \quad
\orient(\degend{i}{n}) = \exmon^{n-i} \mult \exmon^i(\emptyset) : \norient{n+1} \to \norient{n} .
\]
\end{paragraph}

In the remainder of the paper, we will describe explicitly this ``abstract'' cosimplicial object of orientals
and show that it corresponds to the classical one defined by Street in \cite{StreetOrient}.

\subsection{Free expansion on a polygraph}
\label{subsec:free_exp_pol}

We know from Proposition~\ref{prop:free_exp} that the forgetful functor
$\fgf:\expcat\to\ocat$ admits a left adjoint $\frf:\ocat\to\expcat$
taking an \oo-category $C$ to the \ndef{free expansion} $\frf
C=(\exmon C,\orig,\econe)$ on~$C$. We now present 
a concrete description of $\frf C$ in the
particular case where the \oo-category  $C$ is freely generated by a polygraph.  

\begin{paragr}\label{paragr:construct_free_exp}
 Let  $S$ be a polygraph and $C=\free S$ the free \oo-category
generated by $S$.  We shall define an \oo-category $\expan C$, freely
generated by a polygraph, together with
\begin{itemize}
\item an inclusion morphism $\inc:C\to \expan C$,
\item a distinguished $0$-cell $\orig$ of $\expan C$,
\item a morphism $\econe:\expan C\to \cone\expanorig C$,
\end{itemize}
in such a way that, eventually, $(\expan C,\orig,\econe)$
becomes an \oo-category with expansion, and in fact coincides with
$\frf C$.

We now define $\expan C$, $\inc$ and $\econe$ by simultaneous
induction on the dimension. In each dimension~$n$, $\expan C$ will be
defined by freely adjoining a set $R_n$ of new $n$-generators to those
of $C$, together with source and target maps from $R_n$ to $\expan
C_{n-1}$. The morphism $\inc$ will be induced by the natural
inclusion of generators $S_n\to S_n\amalg R_n$. Throughout the construction,
we abuse notations by identifying any $n$-cell $x\in C_n$ with
$\inc x\in\expan C_n$.

\begin{itemize}
\item For $n=0$, $R_0$ is a singleton and $\inc:C_0 \to \expan C_0$ is the
  natural inclusion $S_0 \to S_0 \amalg R_0$.
  The unique element of $R_0$ is denoted
  by $\orig$ and eventually becomes the distinguished $0$-cell of~$\expan C$.
\item For $n=1$, to each $a\in S_0$ corresponds a generator $r_a\in
  R_1$ such that $r_a:\orig\to a$. Now to any generator $a:u\to v$
  in $S_1$ correspond source and target cells $u$ and $v$ in $C_0 \subset
  \expan C_0$.
  Therefore the $1$-cells in $\expan C_1$ are defined as freely generated by
  $S_1\amalg R_1$, and $\inc:C_1\to  \expan C_1$ is induced by the natural
  inclusion $S_1\to S_1\amalg R_1$. 

Moreover the $0$-cells of $\cone\expanorig C$ are now defined as the
$0$-cones of origin $\orig$ in $\expan C$,~that~is,
the $1$-cells $u$ of $\expan C$ with $\sce u_0=\orig$. Finally
  $\econe:\expan C_0\to \cone\expanorig C_0$ is defined by 
  \begin{itemize}
  \item $\econe(a)=r_a$ for each $a\in S_0$,
  \item $\econe(\orig)=\unit{\orig}$.
  \end{itemize}
  In particular, the first degeneracy condition holds.
\item Let $n\geq 1$ and suppose we have defined $\expan C$  together with a morphism $\inc:C\to \expan C$, up to dimension $n$, as
  well as an expanding homotopy $\econe:\expan C\to \cone\expanorig C$ up to
  dimension~$n{-}1$. So we get the following diagram:
  \begin{displaymath}
    \begin{xy}
      \xymatrix{C_{n-1}\ar[d]_{\inc}& \Doubl C_n\ar[d]_{\inc} \\
                       \expan C_{n-1}\ar[d]_{\econe} & \Doubl \expan C_n \\
                        \cone\expanorig C_{n-1}
                        \ar[ur]_{\ppal{-}}&    }
    \end{xy}
  \end{displaymath}
In addition, we suppose that $\expan C_n$ is freely generated by the
set $S_n\amalg R_n$ where
\[
  R_n=\setbis{\econe_a}{a\in S_{n-1}}.
  \]

We must now extend $\expan C$  together with $\inc:C\to\expan C$ up to
dimension $n+1$, and the expanding homotopy $\econe$ up to
dimension $n$. The $(n{+}1)$-generators of
$\expan C$ are twofold:
\begin{itemize}
\item each $a:u\to v$ in $S_{n+1}$ becomes an $(n{+}1)$-generator
 of $\expan C$ with source and target $u$ and $v$ in
 $C_n \subset \expan C_n$;
\item to each $a:u\to v$ in $S_n$ corresponds a new generator $r_a$ in
  $R_{n+1}$. By induction hypothesis, $x=\econe(u)$ and
  $y=\econe(v)$ are two parallel $(n{-}1)$-cones of origin $\orig$ in
  $\expan C$ such that $\Bas{x}=u$ and $\Bas{y}=v$.
Therefore, by~\ref{paragr:cone_formulas}, we may define
  the source and target of~$r_a$ by
  \[
    \begin{aligned}
      \sce{r_a}&=\econe_{\tge a_{n-1}}=\econe_v,
      \\
      \tge{r_a}&=a\comp 0 \econe_{\sce a_0}\comp 1 \cdots
      \comp{n-2} \econe_{\sce a_{n-2}} \comp{n-1} \econe_{\sce a_{n-1}} =
      a \comp 0 \econe_{\sce u_0} \comp 1 \cdots \comp{n-2} \econe_{\sce u_{n-2}} \comp{n-1} \econe_u.
    \end{aligned}
  \]
\end{itemize}  
Thus $\expan C_{n+1}$ is defined as the set of freely generated
$(n{+}1)$-cells over $S_{n+1}\amalg R_{n+1}$, and the natural inclusion
$S_{n+1}\to S_{n+1}\amalg R_{n+1}$ induces $\inc:C_{n+1}\to \expan
C_{n+1}$.
Now, having defined~$\expan C$ up to dimension $n+1$, the $n$-cones in
$\cone\expanorig C$ are determined and the remaining task is to define
$\econe:\expan C_n\to \cone\expanorig C_n$. By
Lemma~\ref{lemma:universal}, it is sufficient to define $\econe$
on the generators of $\expan C_n$, that is, on the elements of
$S_n\amalg R_n$, provided the commutation conditions for source and target
are satisfied. There are two cases to consider:
\begin{itemize}
\item if $a:u\to v$ is in $S_n$, we have defined $r_a\in R_{n+1}$ in
  such a way that $r_a$ is the principal cell of an $n$-cone $z:x\to
  y$ where $x=\econe(u)$, $y=\econe(v)$ and
  $\Bas{z}=a$. Therefore $\EPS z=\econe(\EPS a)$ for $\epsilon=\pm$ so that we may define $\econe(a)=z$;   
\item if $a\in R_n$, by induction hypothesis, $a$ is of the form
  $r_b=\econe_b$ for some $b\in S_{n-1}$. Therefore $\sce a_0=\orig$
  and by~\ref{lemma:unique_deg_cone}, we may define $\econe(a)$ as the
  unique degenerate cone of base $a$. If $n=1$, $\econe(\SCE
  a)=\econe(\orig)=\unit\orig=\SCE\econe(a)$ and $\econe(\TGE
  a)=\econe(b)=\TGE\econe(a)$, which entails compatibility with source
  and target.
  If $n\geq 2$, then
 \[
 \begin{aligned}
   \SCE{a} & =\sce{r_b}=\econe_{\tge b_{n-2}},
   \\
   \TGE{a} & =\tge{r_b}=b\comp 0 \econe_{\sce b_0}\comp 1 \cdots
    \comp{n-2}\econe_{\sce b_{n-2}}.
 \end{aligned}
\]
 By induction, applying $\econe$ to the above equations, and
 using the degeneracy conditions up to dimension $n$ gives
 \[
   \begin{aligned}
   \econe(\SCE a)& =\SCE\econe(a),\\
   \econe(\TGE a)&=\econe(b)\comp 0 \econe(\econe_{\sce b_0})\comp 1 \cdots
    \comp{n-2}\econe(\econe_{\sce b_{n-2}}),
   \end{aligned}
\]
 and it remains to show that $\TGE\econe(a)=\econe(\TGE a)$. First,
 both cones have the same base~$\TGE a$, and by induction, for
 $\epsilon=\pm$,
 \[
 \EPS\econe(\TGE a)=\econe(\EPS\TGE a)=\econe(\EPS\SCE
   a)=\EPS\econe(\SCE a)=\EPS\SCE\econe(a)=\EPS\TGE\econe(a),
 \]
 so that both cones have same source and target. Finally, the
 principal cell of $\TGE\econe(a)$ is just $a$ by
 Lemma~\ref{lemma:unique_deg_cone}, whereas the principal cell of
 $\econe(\TGE a)$ is the one of $\econe(b)\comp 0 \econe(\econe_{\sce b_0})\comp 1 \cdots
    \comp{n-2}\econe(\econe_{\sce b_{n-2}})$ by the above formula. As
    all terms $\ppal{\econe(\econe_{\sce b_i})}$ are identities, one
    gets $\ppal{\econe(\TGE a)}=\ppal{\econe(b)}=a$ and we are done.
\end{itemize}
By construction, $R_{n+1}$ consists in generators of the form
$\econe_a$ where $a\in S_n$, and for all $a\in S_{n+1}\amalg R_{n+1}$,
$\Bas\econe a=a$. Now, for each $u\in \expan C_{n-1} $,
$\econe_{\econe u}=\unit{\econe_u}$: this holds by construction for
generators and extends to any $(n{-}1)$-cell by the functoriality
conditions of~\ref{paragr:exp}.
Therefore $\econe$ satisfies the expanding homotopy
conditions up to dimension $n$, which completes the induction step.
\end{itemize}
\end{paragr} 

\begin{paragr}\label{paragr:def_expan_S}
To sum up, given $C$ a free \oo-category on a polygraph $S$, we have
defined an \oo-category $\expan C$, endowed with a distinguished
$0$-cell $\orig:1\to \expan C$ and an expanding homotopy
$\econe:\expan C\to \cone\expanorig C$. By construction, $\expan
C$ is itself free on a polygraph whose set of $n$-generators is $S_n\amalg R_n$,
where
\[
  R_0 = \{\orig\} \qquad\text{and}\qquad R_n = \{ r_a \mid a
  \in S_{n-1}\} \quad\text{for $n \ge 1$}.
\]
The sources and the targets of the generators in $S_n$ are inherited by
those of $S$ and, if $a$ is in~$S_{n-1}$, for $n \ge 1$, the source and the
target of $r_a$ are
\[
  \sce{r_a}=\econe_{\tge a_{n-1}}
  \qquad\text{and}\qquad
  \tge{r_a}=a\comp 0 \econe_{\sce a_0}\comp 1 \cdots
  \comp{n-1}\econe_{\sce a_{n-1}}
\]
and the expanding homotopy is defined by
\[ \econe_a = r_a, \qquad \econe_\orig = \unit{\orig} \qquad\text{and}\qquad \econe_{r_a} = \unit{r_a}. \]
We will denote this polygraph by $\expan S$, so that $\expan C=\fexpan S$.

\end{paragr}

\begin{proposition}\label{prop:free_exp_pol}
  For any polygraph $S$, the left adjoint $\frf:\ocat\to\expcat$ to
  the functor $\fgf:\expcat\to\ocat$ takes $C = \free S$ to $(\expan C,
  \orig, \econe)$, and therefore the expansion monad $\exmon$ takes
$C=\free S$ to $\expan C = \fexpan S$.
\end{proposition}

\begin{proof}
Let $(D,\orig,\econe)$ be
  an \oo-category with expansion, and $f:\free S\to D$ a morphism in $\ocat$. 
We have to show that there is a unique morphism
\[
\free f:(\expan C,\orig,\econe)\to (D,\orig,\econe)
\]
in $\expcat$ such that the following diagram commutes:
\[
\xymatrix{
\free S \ar[d]_{\inc}\ar[rd]^f & \\
\expan C \ar[r]_{\free f} & D
}
\]
Note that, by abuse of notation, we identify here $\free f$ with $\fgf\free
f$, as $\fgf\free f$ entirely determines $\free f$, as soon as it
commutes with the origins and the expanding homotopies.
The construction is by induction on the dimension.
\begin{itemize}
\item In dimension $n=0$, $\expan S_0=S_0\amalg \set{\orig}$, and we define
  $\free f a= fa$ if $a\in S_0$ and $\free f a=\orig$ if
  $a=\orig$. This is clearly the only possible choice.
\item In dimension $n=1$, $\expan S_1=S_1\amalg R_1$. If $a\in S_1$,
  we set $\free fa=fa$, the commutation with source and target being
  straightforward.   If $a\in R_1$, $a=r_b=\econe_b$ for some $b\in
  S_0$, and we define $\free fa=\econe_{fb}$, which again ensures
  commutation with source and target, and defines $\free f$ up to
  dimension $1$, in the unique possible way. Moreover, for each $u\in\expan
  C_0$, $\econe(\free f u)=\free f(\econe u)$. In fact, either
  $u=\orig$, in which case $\econe_{\free
    f\orig}=\econe_{\orig}=\unit{\orig}=\free f(\unit{\orig})=\free
  f(\econe_\orig)$, or $u\in S_0$, in which case $\econe_{\free f
    u}=\free f(r_u)=\free f(\econe_u)$ by definition.
\item Let $n\geq 1$ and suppose that, up to
  dimension $n$, we have defined a morphism $\free f:\expan C\to D$ such that
$\free f\circ \inc=f$ and $\free f$ commutes to  the
expanding homotopies up to dimension~$n{-}1$.  We extend $\free f$ to dimension $n+1$
by defining first $\free f a$ when $a\in \expan S_{n+1}$. As $\expan S_{n+1}=S_{n+1}\amalg R_{n+1}$, there are two cases
to consider:
\begin{itemize}
\item if $a\in S_{n+1}$, we take $\free f a=fa$, and the commutation
  with source and target is straightforward. The condition $\free
  f\circ \inc=f$ implies that this choice is unique;
\item if $a\in R_{n+1}$, $a$ is of the form $r_b$ with $b\in S_n$. By
  induction hypothesis, we already have an $n$-cell $\free f b\in
  D_n$. Now $D$ is endowed with an expanding homotopy $\econe$ so
  that we get an $n$-cone $x=\econe(\free f b)$, whose principal cell
  $u=\ppal x$ is an $(n{+}1)$-cell of~$D$. Thus, we may define $\free f
  a=u$. The construction of $r_b$ given
  in~\ref{paragr:construct_free_exp} ensures the commutation with
  source and target. Moreover, because $\free f$ must commute to the
  expanding homotopies, the above choice for $\free f r_b$ is unique.
\end{itemize}
By Lemma~\ref{lemma:universal}, the above values determine a unique
extension of the morphism $\free f$ in dimension $n+1$.
It remains to check that the morphism $\free f$ so defined actually commutes to expanding
homotopies up to dimension $n$, that is, $\econe(\free f u)=\free
f\econe(u)$ for all $u\in\expan C_n$. By functoriality, it suffices
to check this commutation on generators. Thus, if $a\in S_n$,
$\econe_a=r_a\in R_{n+1}$ and $\free f\econe_a=\econe_{\free f a}$ by
definition. If $a\in R_n$, $a=r_b=\econe_b$ for some $b\in S_{n-1}$,
and the degeneracy conditions together with the induction hypothesis
yield
\[
  \econe_{\free f
  a}=\econe_{\free f\econe_b}=\econe_{\econe_{\free f b}}=\unit{\econe_{\free f b}}=\unit{\free
  f\econe_b}=\free f\unit{\econe_b}=\free f\econe_{\econe_b}=\free
f\econe_a.
\qedhere
  \]
\end{itemize}
\end{proof}

\begin{remark}\label{rem:exist_adj}
  Proposition~\ref{prop:free_exp}, whose proof is based on  an abstract argument using
  the theory of sketches, states that the forgetful functor
  $\fgf:\expcat\to\ocat$ admits a left adjoint. Proposition~\ref{prop:free_exp_pol} gives an alternate
  proof. Indeed, it shows that this forgetful functor admits a left adjoint
  relative to the subcategory of \oo-categories freely generated by a
  polygraph. As this subcategory contains a small dense subcategory (for
  instance, the category of globular pasting schemes indexing operations of
  \oo-categories), this implies that the forgetful functor admits a left
  adjoint (provided that we know that the category $\expcat$ is cocomplete).
\end{remark}

\begin{corollary}\label{cor:free_orientals}
  For each $n\geq 0$, the $n$-th oriental $\norient n$ is a free
  \oo-category on a polygraph.
\end{corollary}

\begin{proof}
  By induction on $n$. As $\norient 0$ is the terminal
  \oo-category, it is freely generated by the polygraph having a single
  $0$-generator and no generator of higher dimensions. Let $n\geq 0$
  and suppose $\norient n$ is free on a polygraph. By definition,
  $\norient{n+1}=\exmon(\norient n)=\expan{\norient{n}}$, which is again free on a
  polygraph by~Proposition~\ref{prop:free_exp_pol}.
\end{proof}

\begin{proposition}\label{prop:monad_action_functor}
  Let $C=\free S$, where $S$ is a polygraph, and let $f:C\to D$ be an
  \oo-functor. Then the action of the \oo-functor $\exmon f : \exmon C \to \exmon D$ on generators is given by
  \[ (\exmon f)(a) = \inc f(a)\quad \text{for $a$ in $S_n$ with $n \ge 0$}, \]
  \[ (\exmon f)(\orig) = \orig
    \qquad\text{and}\qquad
    (\exmon f)(r_a) = \econe_{\inc f(a)}\quad
    \text{for $a \in S_{n-1}$ with $n \ge 1$,}
  \]
  where $\orig$ and $\econe$ respectively denote the origin and the
  expanding homotopy in $\exmon C$ and $\exmon D$. 
\end{proposition}

\begin{proof}
  Let $g=\inc f:\free S\to \exmon D$. By naturality of $\inc$, the following diagram commutes:
  \[
    \xymatrix{\free S\ar[r]^f \ar[d]_{\inc}\ar[rd]^g& D\ar[d]^{\inc}\\
    \exmon C\ar[r]_{\exmon f} & \exmon D}
\]
Therefore, by Proposition~\ref{prop:free_exp_pol}, $\exmon f$ is the
unique \oo-morphism commuting to origins and expanding homotopies
making the bottom-left triangle commute, and the equations follow from
the description of $\expan S$.
\end{proof}

By Proposition~\ref{prop:free_exp_pol}, if $C$ is freely generated by a
polygraph, then the unit of the expansion monad is given by the morphism
$\inc:C\to \expan C$ of paragraph~\ref{paragr:construct_free_exp}. We
end the section by a description of the multiplication of the monad:

\begin{proposition} \label{prop:mult}
  Let $C$ be an \oo-category freely generated by a polygraph. Then the
  multiplication of the expansion monad $\mu : \exmon^2 C \to \exmon C$ is the
  \oo-functor defined on generators by
  \[
    \mu(\eta x) = x,
    \qquad
    \mu(\orig) = \orig
    \qquad\text{and}\qquad
    \mu(\xi_{\eta x}) = \xi_x,
  \]
  where $x$ is a generator of $\exmon C = \expan C$.
\end{proposition}

\begin{proof}
  The first equation holds for any monad. As for the other ones, they follow
  from the fact that $\mu = \fgf \epsilon \frf$ is induced by a morphism of
  \oo-categories with expansion.
\end{proof}

\section{Oriental calculus}
\label{sec:calculus}

\subsection{Syntax for expansion}

We consider an \oo-category with expansion $(C, \orig, \econe)$.

\begin{paragr} \label{paragr:chevron}
For any $n$-cell $x$ in $C$, we write $\ch x$ for the $(n{+}1)$-cell $\econe_x$ defined in paragraph~\ref{paragr:exp},
which we call \emph{chevron} of $x$.
\begin{itemize}
\item It has the following source and target:
\[
\ch x : \orig \to x \mbox{ if } n = 0, \qquad
\ch x : \ch{\tge x_{n-1}} \to x \comp 0 \ch{\sce x_0} \comp 1 \cdots \comp {n-1} \ch{\sce x_{n-1}} \mbox{ if } n > 0.
\]
\item It is the principal cell of the $n$-cone $\econe(x) : \orig \cto x$ given by the following cells:
\[
\ch{\sce x_0}, \ch{\tge x_0}, \ldots, \ch{\sce x_{n-1}}, \ch{\tge x_{n-1}}, \ch x.
\]
\end{itemize}
\end{paragr}

\begin{examples}
Starting from $\ch x : \orig \to x$ for $n = 0$, we get $\ch x : \ch{\tge x_0} \to x \comp 0 \ch{\sce x_0} : \orig \to \tge x_0$ for $n = 1$, 
and $\ch x : \ch{\tge x_1} \to x \comp 0 \ch{\sce x_0} \comp 1 \ch{\sce x_1} : \ch{\tge x_0} \to \tge x_1 \comp 0 \ch{\sce x_0} : \orig \to \tge x_0$ for $n = 2$.
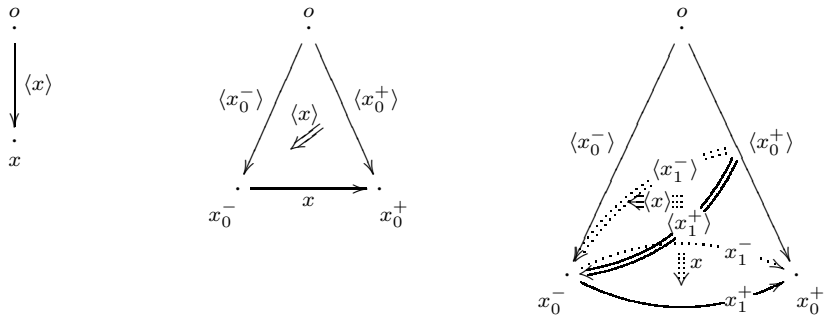
\begin{displaymath}
  \begin{xy}
   \xymatrix{\overset{\orig}{\cdot} \ar[d]^{\ch{x}}\\
                       \underset{x}{\cdot}}
  \end{xy}
\hspace{5em}
  \begin{xy}
    \xymatrix @=1.5em{
&\overset{\orig}{\cdot} 
           \ar[ldd]_{\ch{\sce x_0}} 
            \ar[rdd]^{\ch{\tge x_0}}
         &
        \ar@2{->}(11,-13)*{};(7,-16)*{}_(.1)*-<4pt>{\scriptstyle{\ch{x}}}\\
&&\\
\underset{\makebox[0pt]{$\scriptstyle{\sce x_0}\ \ \ $}}{\cdot} 
\ar[rr]_x
&&
\underset{\makebox[0pt]{$\ \ \ \scriptstyle{\tge x_0}$}}{\cdot}  }
  \end{xy}
\hspace{5em}
  \begin{xy}
    \xymatrix @=1.2em{
&&
\overset{\orig}{\cdot} 
          \ar[lldddd]_{\ch{\sce x_0}} 
           \ar[rrdddd]^{\ch{\tge x_0}}
            \ar@2{.>}(15,-30)*{};(15,-34)*{}^(.3){x}
&&\\
&&&& \\
&&&         
             \ar@2@/^1pc/[ddlll]|(.4){\ch{\tge x_1}}
              \ar@2{.>}@/_.8pc/[ddlll]|(.3){\ch{\sce x_1}}
              \ar@3{.>}(15,-23)*{};(8,-23)*{}|(.45){\ch{x}}&
\\
&&&&
\\
     \underset{\makebox[0pt]{$\scriptstyle{\sce x_0}\ \ \ $}}{\cdot}
              \ar@{.>}@/^1pc/[rrrr]|(.75){\sce x_1}
               \ar@/_1pc/[rrrr]|(.75){\tge x_1}
&&&&
      \underset{\makebox[0pt]{$\ \ \ \scriptstyle{\tge x_0}$}}{\cdot}  }
  \end{xy}
\end{displaymath}
\end{examples}

\begin{paragr}
The four axioms of paragraph~\ref{paragr:exp} can be rewritten as follows for any $n > p$,
for any $n$-cells $x, y$ such that $\tge x_p = \sce y_p$, and for any cell $u$:
\[
\ch{y \comp p x} = \tge y_{p+1} \comp 0 \ch{\sce x_0} \comp 1 \cdots \comp{p-1} \ch{\sce x_{p-1}} \comp p \ch x \comp{p+1} \ch y, \quad
\ch{\unit u} = \unit{\ch u}, \quad
\cch u = \unit{\ch u}, \quad
\ch \orig = \unit \orig.
\]
\end{paragr}

\begin{remarks}\ %
\begin{itemize}
\item In case $n = p+1$, we get $\tge y_{p+1} = y$, so that our first axiom can be rewritten as follows:
\[
\ch{y \comp p x} = y \comp 0 \ch{\sce x_0} \comp 1 \cdots \comp{p-1} \ch{\sce x_{p-1}} \comp p \ch x \comp{p+1} \ch y.
\]
\item The last axiom has a single occurrence:
\begin{displaymath}
 \begin{xy}
   \xymatrix{\overset{\orig}{\cdot} \ar[d]^{\ch \orig \, = \, \unit \orig}\\
                       \underset{\orig}{\cdot}}
  \end{xy}
\end{displaymath}
If we write $\orig = \ch \jok$ where $\jok$ is an extra cell of dimension $-1$,
this axiom becomes a particular case of the previous one : $\cch \jok = \unit{\ch \jok}$.
We shall not introduce such a cell but we shall use a similar idea in our \emph{simplicial notation} for generators of orientals.
\end{itemize}
\end{remarks}

\begin{examples}\ %
\begin{itemize}
\item If $x, y$ are 1-cells such that $\tge x_0 = \sce y_0$, we get $\ch{y \comp 0 x} = y \comp 0 \ch x \comp 1 \ch y$.
\item If $x, y$ are 2-cells such that $\tge x_0 = \sce y_0$, we get $\ch{y \comp 0 x} = \tge y_1 \comp 0 \ch x \comp 1 \ch y$.
\item If $x, y$ are 2-cells such that $\tge x_1 = \sce y_1$, we get $\ch{y \comp 1 x} =  y \comp 0 \ch{\sce x_0} \comp 1 \ch x \comp 2 \ch y$.
\end{itemize}
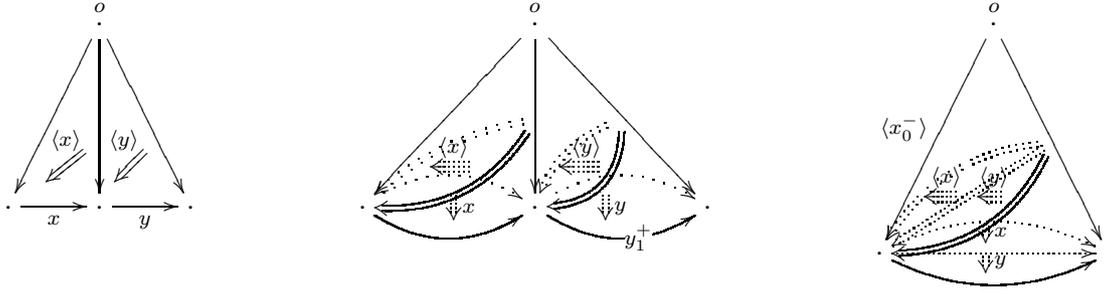
\begin{displaymath}
  \begin{xy}
    \xymatrix{
&\overset{\orig}{\cdot} 
        \ar[ldd]
        \ar[dd] 
        \ar[rdd]
       &
      \ar@2{->}(10,-17)*{};(5,-21)*{}_(.1)*-<4pt>{\scriptstyle{\ch{x}}}
       \ar@2{->}(18,-17)*{};(14,-21)*{}_(.1)*-<4pt>{\scriptstyle{\ch{y}}}\\
&&\\
\cdot 
\ar[r]_x
& \cdot
\ar[r]_y
&\cdot}
  \end{xy}
\hspace{5em}
  \begin{xy}
    \xymatrix{
&&
     \overset{\orig}{\cdot} 
          \ar[dd]
           \ar[rrdd]
           \ar[lldd]
            \ar@2{.>}(12,-23)*{};(12,-26)*{}^{x}
            \ar@2{.>}(32,-23)*{};(32,-26)*{}^{y}
&&\\
& &         \ar@2@/^{1pc}/[dll]
              \ar@2{.>}@/_/[dll]
              \ar@3{.>}(14,-19)*{};(9,-19)*{}_(.4){\ch{x}}
               \ar@3{.>}(31,-19)*{};(26,-19)*{}_(.3){\ch{y}}
&
\ar@2@/^{1pc}/[dl]
              \ar@2{.>}@/_/[dl]
& \\
           \cdot
              \ar@{.>}@/^1pc/[rr]
               \ar@/_1pc/[rr]
&&
      \cdot  
\ar@{.>}@/^1pc/[rr]
               \ar@/_1pc/[rr]|(.6){\tge y_1}
&& \cdot }
  \end{xy}
\hspace{5em}
  \begin{xy}
    \xymatrix @=1.2em{
&&
\overset{\orig}{\cdot} 
          \ar[lldddd]_{\ch{\sce x_0}}
           \ar[rrdddd]
             \ar@2{.>}(14,-27)*{};(14,-29)*{}^(.3){x}
            \ar@2{.>}(14,-31)*{};(14,-33)*{}^(.3){y}
&&\\
&&&& \\
&&&         
             \ar@2@/^{1.1pc}/[ddlll]
              \ar@2{.>}@/_{.9pc}/[ddlll]
               \ar@2{.>}[ddlll]
\ar@3{.>}(16,-23)*{};(13,-23)*{}_(.3){\ch{y}}
\ar@3{.>}(10,-23)*{};(6,-23)*{}_(.3){\ch{x}}              
&
\\
&&&&
\\
     \cdot
              \ar@{.>}@/^{1pc}/[rrrr]
               \ar@{.>}[rrrr]
               \ar@/_{1pc}/[rrrr]
&&&&
     \cdot  }
  \end{xy} \vspace{1ex}
\end{displaymath}
\begin{itemize}
\item If $u$ is a 0-cell, we have $\unit u : u \to u$ and $\ch u : \orig \to u$. So we get 
\[
\ch{\unit u} = \unit{\ch u} : \ch u \to \unit u \comp 0 \ch u = \ch u \mbox{ and }
\cch u = \unit{\ch u} : \ch u \to \ch u \comp 0 \ch \orig = \ch u \comp 0 \unit \orig = \ch u.
\]
\begin{displaymath}
   \begin{xy}
    \xymatrix @=1.3em{
&&\overset{\orig}{\cdot} 
           \ar[llddd]_{\ch u} 
            \ar[rrddd]^{\ch u}
         &&
        \ar@2{->}(20,-14)*{};(9,-21)*{}|{\ch{\unit u} \, = \, \unit{\! \ch u}}\\
&&&&\\
&&&&\\
\underset{\makebox[0pt]{$\scriptstyle{u}\ \ \ $}}{\cdot} 
\ar[rrrr]_{\unit{u}}
&&&&
\underset{\makebox[0pt]{$\ \ \ \scriptstyle{u}$}}{\cdot}  }
  \end{xy}
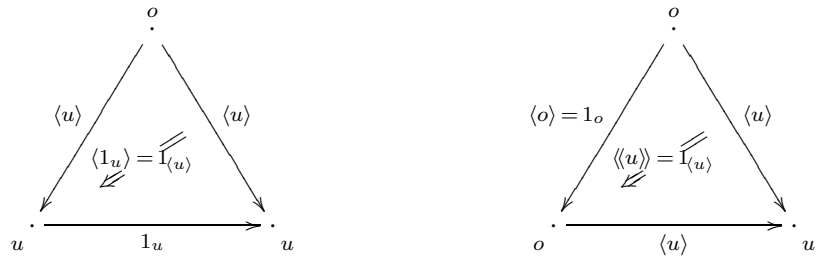
\hspace{8em}
  \begin{xy}
    \xymatrix @=1.3em{
&&\overset{\orig}{\cdot} 
           \ar[llddd]_{\ch \orig \, = \, \unit{\orig}} 
            \ar[rrddd]^{\ch u}
         &&
        \ar@2{->}(20,-14)*{};(9,-21)*{}|{\cch u \, = \, \unit{\! \ch u}}\\
&&&&\\
&&&&\\
\underset{\makebox[0pt]{$\scriptstyle{\orig}\ \ \ $}}{\cdot} 
\ar[rrrr]_{\ch u}
&&&&
\underset{\makebox[0pt]{$\ \ \ \scriptstyle{u}$}}{\cdot}  }
  \end{xy}
\end{displaymath}

\item If $u$ is a 1-cell, we have $\unit u : u \to u : \sce u_0 \to \tge u_0$ and
$\ch u : \ch{\tge u_0} \to u \comp 0 \ch{\sce u_0} : \orig \to \tge u_0$.
So~we~get $\ch{\unit u} = \unit{\ch u} : \ch u \to \unit u \comp 0 \ch{\sce u_0} \comp 1 \ch u = \ch u$
and
\[
\begin{array}{c}
\cch u = \unit{\ch u} : \ch{\! \tge{\ch u}_1} = \ch{u \comp 0 \ch{\sce u_0} \!} = u \comp 0 \cch{\sce u_0} \comp 1 \ch u =
u \comp 0 \unit{\ch{\sce u_0}} \comp 1 \ch u = \ch u \to \\
\ch u \comp 0 \ch{\! \sce{\ch u}_0} \comp 1 \ch{\! \sce{\ch u}_1} = \ch u \comp 0 \ch \orig \comp 1 \cch{\tge u_0} =
\ch u \comp 0 \unit \orig \comp 1 \unit{\ch{\tge u_0}} = \ch u.
\end{array}
\]
\begin{displaymath}
  \begin{xy}
    \xymatrix @=1.2em{
&&
\overset{\orig}{\cdot} 
          \ar[lldddd]_{\ch{\sce u_0}} 
           \ar[rrdddd]^{\ch{\tge u_0}}
            \ar@2{.>}(15,-30)*{};(15,-34)*{}^(.3){\unit{u}}
&&\\
&&&& \\
&&&         
             \ar@2@/^1.3pc/[ddlll]|(.4){\ch u}
              \ar@2{.>}@/_.8pc/[ddlll]|(.2){\ch u}
              \ar@3{.>}(17,-23)*{};(8,-23)*{}_{\ch{\unit u} \, = \, \unit{\! \ch u}}&
\\
&&&&
\\
     \underset{\makebox[0pt]{$\scriptstyle{\sce u_0}\ \ \ $}}{\cdot}
              \ar@{.>}@/^1pc/[rrrr]|(.7){u}
               \ar@/_1pc/[rrrr]|(.7){u}
&&&&
      \underset{\makebox[0pt]{\ \ \ $\scriptstyle{\tge u_0}$}}{\cdot}}
  \end{xy}
\hspace{8em}
\begin{xy}
    \xymatrix @=2em{
&&\overset{\orig}{\cdot}
\ar[lldddd]_{\ch \orig \, = \, \unit{\orig}}
\ar[rrrddd]^{\ch{\tge u_0}}
\ar[rddddd]|(.7){\ch{\sce u_0}}
&&& \\
&&&&& \\
&&&
\ar@2{.>}@<-6pt>[ld]|{\cch{\tge u_0} \, = \, \unit{\! \ch{\tge u_0}}}
\ar@2{->}@<8pt>[d]^{\ch u}
&& \\
&&
\ar@2{->}[ld]|{\cch{\sce u_0} \, = \, \unit{\! \ch{\sce u_0}}}
&
\ar@3{.>}@<4pt>[l]^(.4){\cch u \, = \, \unit{\! \ch u}}
&&{\cdot}\ {\scriptstyle\tge u_0}
\\
{\scriptstyle\orig}\ \cdot
\ar[rrrd]_{\ch{\sce u_0}}
\ar@{.>}[rrrrru]|(.75){\ch{\tge u_0}}&&
\ar@2{}@<-1pt>[rd]_(-0.10){}="s"
\ar@2{}@<-1pt>[rd]_(0.60){}="t"
\ar@2{.>}@<-1pt>"s";"t"|{\ch u}
&&& \hbox to 2mm{\hfill\vbox to 10mm{}}\\
&&&\underset{\sce u_0}{\cdot}
\ar[rruu]_u&&
}
\end{xy}
\end{displaymath}
\end{itemize}
\end{examples}

\subsection{Syntax for orientals}

In paragraph~\ref{paragr:construct_free_exp}, the unit $\munit : C \to \expan C$ of the expansion monad is considered as an inclusion,
but since orientals are obtained by iterating this monad, our \emph{simplicial notation} for orientals uses an explicit \emph{shift}.

\begin{paragr}
Let us introduce the following notations:
\begin{itemize}
\item If $x$ is an $m$-cell in $\norient n$, we write $\shift x$, called \emph{shift} of $x$, for the $m$-cell $\munit(x)$ in $\norient{n+1}$.
\item We write $\smp 0$ for the origin $\orig$, which is a $0$-cell in $\norient n$ for any $n$.
\item More generally, if $0 \leq i \leq n$, we write $\smp i$ for the $0$-cell $\munit^i(\orig)$ in $\norient n$.
\item If $x$ is an $m$-cell in $\norient n$, we write $\sch 0 x$ for the chevron $\ch x$, which is an $(m+1)$-cell in~$\norient n$.
\item More generally, if $0 \leq i \leq n$ and $x'$ is an $m$-cell in $\norient n$ of the form $\munit^i(x)$ for some $x$~in~$\norient{n-i}$,
we write~$\sch i {x'}$ for the $(m+1)$-cell $\munit^i \ch x$ in $\norient n$.
\end{itemize}
\end{paragr}

\begin{remarks}\ %
\begin{itemize}
\item Shift is a notation for the embedding $\munit : \norient n \emb \expan{\norient n} = \norient{n+1}$
induced by the map $i \mapsto i+1$ from $\deltan n  = \set{0, \ldots, n}$ to $\deltan{n+1} = \set{0, \ldots, n+1}$,
which must not be confused with the \emph{canonical inclusion} $\norient n \subset \norient {n+1}$
induced by the inclusion $\deltan n \subset \deltan{n+1}$.
\item In practice, $\shift x$ is obtained by incrementing all integers occurring in $x$,
or more precisely, by applying the following rules:
\[
\shift{\smp i} = \smp{i{+}1}, \qquad
\shift{\sch i x} = \sch{i{+}1}{\shift x}, \qquad
\shift{y \comp p x} = \shift y \comp p \shift x, \qquad
\shift{\unit u} = \unit{\shift u}.
\]
\end{itemize}
\end{remarks}

\begin{paragr}
By paragraph~\ref{paragr:def_expan_S}, the oriental $\norient n$ has a $0$-generator $\smp 0$,
and any $m$-generator $s$ of $\norient n$ yields two generators of~$\norient{n+1}$:
\begin{itemize}
\item a \emph{shifted} $m$-generator $s' = \shift s$ standing for $\munit(s)$, with shifted source and target if $m > 0$,
\item an \emph{expanded} $(m{+}1)$-generator $\sch 0 {s'}$, which is just a new notation for the chevron $\ch{s'}$,
whose source and target are given by the same formulas as in paragraph~\ref{paragr:chevron}.
\end{itemize}
\end{paragr}

\begin{paragr} \label{paragr:generators-On}
By induction on $n$, we get that any 0-generator of $\norient n$ is of the form $\smp i$ with $0 \leq i \leq n$,
and more generally, any $m$-generator of $\norient n$ is of the form
\[
\sch{i_0}{\! \sch{i_1}{\ldots, \! \smp{i_m} \cdots} \!} \mbox{ with } 0 \leq i_0 < i_1 < \cdots < i_m \leq n.
\]
In other words, the set of $m$-generators of $\norient n$ is in canonical bijection
with the set of injective order-preserving maps from $\deltan m$ to $\deltan n$.

We write $\smp{i_0, i_1, \ldots, i_m}$ for the above $m$-generator,
which is in fact defined by induction on~$m$:
\[
\smp{i_0, i_1, \ldots, i_m} = \sch{i_0}{\! \smp{i_1, \ldots, i_m} \!}.
\]
In particular, we get a single generator $\ppal{\norient n} = \smp{0, 1, \ldots, n}$ of maximal dimension $n$,
which is called the \emph{principal generator} of $\norient n$.
\end{paragr}

\begin{paragr}
The above notation extends to the case of a \emph{nondecreasing} sequence $i_0 \leq  i_1 \leq \cdots \leq i_m$.
This is easily seen by induction on $m$, since we get 
$\smp{i_0, i_1, \ldots, i_m} = \munit^{i_0} \sch 0 {\! \smp{i_1 - i_0, \ldots, i_m - i_0} \!}$,
but this defines a generator only if the sequence is (strictly) increasing. Otherwise, we get a~unit.
For instance, if $0 = i_0 = i_1 \leq i_2 \cdots \leq i_m$, then we get the following unit:
\[
\smp{0, 0, i_2, \ldots, i_m} = \sch 0 {\! \sch 0 {\! \smp{i_2, \ldots, i_m} \!} \!} = \unit{\smp{0, i_2, \ldots, i_m}}.
\]
The second equality follows indeed from the degeneracy axiom $\cch u= \unit{\ch u}$
where $u = \smp{i_2, \ldots, i_m}$ and $\ch u$ is written $\sch 0 u$.
\end{paragr}

\begin{examples}\ %
\begin{itemize}
\item $\norient 0 = \expan \emptyset$ has a single 0-generator $\smp 0$.
\item $\norient 1 = \expan{\norient 0}$ has the generator of $\norient 0$ and the following ones:
\begin{itemize}
\item the 0-generator $\shift{\smp 0} = \smp 1$,
\item the 1-generator $\schsmp 0 1 = \smp{0,1} : \smp 0 \to \smp 1$.
\end{itemize}
\item $\norient 2 = \expan{\norient 1}$ has the generators of $\norient 1$ and the following ones: 
\begin{itemize}
\item the 0-generator $\shift{\smp 1} = \smp 2$,
\item the 1-generator $\schsmp 0 2 = \smp{0,2} : \smp 0 \to \smp 2$,
\item the 1-generator $\shift{\smp{0,1}} = \smp{1,2} : \smp 1 \to \smp 2$,
\item the 2-generator $\schsmp 0 {1,2} =\smp{0,1,2} : \smp{0,2} \to \smp{1,2} \comp 0 \smp{0,1}$.
\end{itemize}
For $s = \smp{1,2} : \smp 1 \to \smp 2$, we get indeed
\[
\sch 0 s : \sch 0 {\tge s_0} = \schsmp 0 2 = \smp{0,2} \to s \comp 0 \sch 0 {\sce s_0} = \smp{1,2} \comp 0 \schsmp 0 1 = \smp{1,2} \comp 0 \smp{0,1}.
\]
\item $\norient 3 = \expan{\norient 2}$ has the generators of $\norient 2$ and the following ones:
\begin{itemize}
\item the 0-generator $\shift{\smp 2} = \smp 3$,
\item the 1-generator $\smp{0,3} : \smp 0 \to \smp 3$,
\item the 1-generator $\shift{\smp{0,2}} = \smp{1,3} : \smp 1 \to \smp 3$,
\item the 2-generator $\smp{0,1,3} : \smp{0,3} \to \smp{1,3} \comp 0 \smp{0, 1}$,
\item the 1-generator $\shift{\smp{1,2}} = \smp{2,3} : \smp 2 \to \smp 3$,
\item the 2-generator $\smp{0,2,3} : \smp{0,3} \to \smp{2,3} \comp 0 \smp{0,2}$,
\item the 2-generator $\shift{\smp{0,1,2}} = \smp{1,2,3} : \smp{1,3} \to \smp{2,3} \comp 0 \smp{1,2}$,
\item the 3-generator $\smp{0,1,2,3} : \smp{2,3} \comp 0 \smp{0,1,2} \comp 1 \smp{0,2,3} \to \smp{1,2,3} \comp 0 \smp{0,1} \comp 1 \smp{0,1,3}$.
\end{itemize}
For $s = \smp{1,2,3} : \smp{1,3} \to \smp{2,3} \comp 0 \smp{1,2} : \smp 1 \to \smp 3$, we get indeed
\[
\begin{array}{c}
\sch 0 s : \sch 0 {\tge s_1} = \sch 0 {\! \smp{2,3} \comp 0 \smp{1,2} \!} = \smp{2,3} \comp 0 \schsmp 0 {1,2} \comp 1 \schsmp 0 {2,3} =
\smp{2,3} \comp 0 \smp{0,1,2} \comp 1 \smp{0,2,3} \to \\
s \comp 0 \sch 0 {\sce s_0} \comp 1 \sch 0 {\sce s_1} = \smp{1,2,3} \comp 0 \schsmp 0 1 \comp 1 \schsmp 0 {1,3}
= \smp{1,2,3} \comp 0 \smp{0,1} \comp 1 \smp{0,1,3}.
\end{array}
\]
\end{itemize}
\begin{displaymath}
  \overset{\smp{0}}{\cdot}
\hspace{4em}
  \begin{xy}
   \xymatrix{\overset{\smp{0}}{\cdot} \ar[d]^{\smp{0,1}}\\
                       \underset{\smp 1}{\cdot}}
  \end{xy}
\hspace{4em}
  \begin{xy}
    \xymatrix @=1.5em{
&\overset{\smp{0}}{\cdot} 
           \ar[ldd]_{\smp{0,1}} 
            \ar[rdd]^{\smp{0,2}}
         &
        \ar@2{->}(13,-11)*{};(5,-18)*{}|{\scriptstyle{\smp{0,1,2}}}\\
&&\\
\underset{\makebox[0pt]{$\scriptstyle{\smp 1}\ \ \ $}}{\cdot} 
\ar[rr]_{\smp{1,2}}
&&
\underset{\makebox[0pt]{$\ \ \ \scriptstyle{\smp 2}$}}{\cdot}  }
  \end{xy}
\hspace{4em}
  \begin{xy}
    \xymatrix @=1.2em{
&&\overset{\smp{0}}{\cdot}
\ar[lldddd]_{\smp{0,1}}
\ar[rrrddd]^{\smp{0,3}}
\ar[rddddd]|(.2){\smp{0,2}}
&&& \\
&&&&& \\
&&&
\ar@2{.>}@<-5pt>[ld]_(.7)*-<3pt>{\scriptstyle{\smp{0,1,3}}}
\ar@2{->}@<3pt>[d]|{\smp{0,2,3}}
&& \\
&&&
\ar@3{.>}@<4pt>[l]^(.4){\smp{0,1,2,3}}
&&{\hbox to 2mm{\hfill\vbox to 2mm{}}\cdot\scriptstyle{\smp 3}}
\\
*{\scriptstyle{\smp 1}\textstyle{\cdot}\hbox to 3mm{\hfill\vbox to 4mm{}}}
\ar[rrrd]_{\smp{1,2}}
\ar@{.>}[rrrrru]|(.75){\smp{1,3}}&&
\ar@2{.>}[rd]!<-3ex,4ex>^*-<3pt>{\scriptstyle{\smp{1,2,3}}}
\ar@2{->}@<-6pt>[l]_*-<3pt>{\scriptstyle{\smp{0,1,2}}}&&& \hbox to 2mm{\hfill\vbox to 10mm{}}\\
&&&\underset{\smp 2}{\cdot}
\ar[rruu]_{\smp{2,3}}&&
}
  \end{xy}
\end{displaymath}
\end{examples}

\begin{remarks}\ %
\begin{itemize}
\item We have a canonical inclusion $\norient n \subset \norient{n+1}$, so that once the generators of $\norient n$ are known,
it suffices to give the \emph{last generation} (of generators) of $\norient{n+1}$.
The latter is obtained by applying $s \mapsto s' = \shift s$ and then $s' \mapsto \sch 0 {s'}$ to the last generation (of generators) of $\norient n$.
In particular, it does not contain the origin $\smp 0$.
\item Once we have computed the source $u$ and the target $v$ of the principal generator $\ppal{\norient m} $,
we get the source and target of any other $m$-generator $\smp{i_0, \ldots, i_m}$ of $\norient n$ for $n > m$
by applying the substitution $0 \mapsto i_0$, \ldots, $m \mapsto i_m$ to $u$ and to $v$.
In other words, we apply the functor $\orient : \scat \to \ocat$, which is described in the next subsection, to this substitution,
seen as an injective order-preserving map from $\deltan m$ to $\deltan n$.
\end{itemize}
\end{remarks}

\subsection{Expansion monad on orientals}

\begin{paragr} \label{paragr:explicit-monad}
By construction, the expansion monad restricts to orientals:
\begin{itemize}
\item It maps the oriental $\norient n$ to the oriental $\expan{\norient n} = \norient{n+1}$, and the \oo-functor $f : \norient n \to \norient{n'}$
to the \oo-functor $\expan f : \expan{\norient n} = \norient{n+1} \to \expan{\norient{n'}} = \norient{n'+1}$ defined as follows by proposition~\ref{prop:monad_action_functor}:
\[
\expan f \smp 0 = \smp 0, \qquad
\left\{
\begin{array}{r@{\mbox{} = \mbox{}}l}
\expan f \shift s & \shift{f \, s}, \vspace{1ex} \\
\expan f \sch 0 {\shift s} & \sch 0 {\shift{f \, s}}
\end{array}
\right. \mbox{for any $m$-generator $s$ of } \norient n.
\]
\item Its \emph{unit} is the \oo-functor $\munit: \norient n \emb \expan{\norient n} = \norient{n+1}$ defined as follows:
\[
\munit(s) = \shift s \mbox{ for any $m$-generator $s$ of } \norient n.
\]
\item Its \emph{multiplication} is the \oo-functor $\mult : \eexpan{\norient n} = \norient{n+2} \to \expan{\norient n} = \norient{n+1}$
defined as follows by proposition~\ref{prop:mult}:
\[
\mult{\smp 0} = \smp 0, \qquad
\left\{
\begin{array}{r@{\mbox{} = \mbox{}}l}
\mult \shift s & s, \vspace{1ex} \\
\mult \sch 0 {\shift s} & \sch 0 s
\end{array}
\right. \mbox{for any $m$-generator $s$ of } \norient{n+1}.
\]
\end{itemize}
More explicitly, we get the following formulas for $\munit$ and $\mult$:
\[
\begin{array}{c}
\munit \smp{i_0, \ldots, i_m} = \smp{i_0{+}1, \ldots, i_m{+}1} \mbox{ for } 0 \leq i_0 < \cdots < i_m \leq n, \vspace{2ex} \\
\left\{
\begin{array}{r@{\mbox{} = \mbox{}}l}
\mult \smp{i_0{+}1, \ldots, i_m{+}1} & \smp{i_0, \ldots, i_m}, \vspace{1ex} \\
\mult \smp{0, i_0{+}1, \ldots, i_m{+}1} & \smp{0, i_0, \ldots, i_m}
\end{array}
\right. \mbox{for } 0 \leq i_0 < \cdots < i_m \leq n + 1.
\end{array}
\]
\end{paragr}

\begin{remarks}\ %
\begin{itemize}
\item Unlike $\munit$, the \oo-functor $\mult$ is not rigid, since we get the following degenerate case for~$i_0 = 0$:
\[
\mult \smp{0, 1, i_1{+}1, \ldots, i_m{+}1} = \smp{0, 0, i_1, \ldots, i_m} = \unit{\smp{0, i_1, \ldots, i_m}}.
\]
\item Our formulas for $\mult$ and $\munit$ are \emph{generic}, since they do not depend on the dimension $n$ of $\norient n$.
For that reason, objects are omitted in our notation for those natural transformations.
\end{itemize}
\end{remarks}

\begin{example}
The \oo-functor $\mult : \norient 3 \to \norient 2$ is defined as follows:
\[
\begin{array}{c}
\mult \smp 0 = \mult \smp 1 = \smp 0, \quad
\mult \smp 2 = \smp 1, \quad
\mult \smp 3 = \smp 2, \vspace{1ex} \\
\mult \smp{0,1} = \unit{\smp 0}, \quad
\mult \smp{0,2} = \mult \smp{1,2} = \smp{0,1}, \quad
\mult \smp{0,3} = \mult \smp{1,3} = \smp{0,2}, \quad
\mult \smp{2,3} = \smp{1,2}, \vspace{1ex} \\
\mult \smp{0,1,2} = \unit{\smp{0,1}}, \quad
\mult \smp{0,1,3} = \unit{\smp{0,2}}, \quad
\mult \smp{0,2,3} = \mult \smp{1,2,3} = \smp{0,1,2}, \vspace{1ex} \\
\mult \smp{0,1,2,3} = \unit{\smp{0,1,2}}.
\end{array}
\]
\end{example}

\begin{proposition} \label{prop:cosimplicial}
The cosimplicial object of orientals $\orient : \scat \to \ocat$ maps the object $\deltan n$ to the oriental $\norient n$,
and the order-preserving map $\phi : \deltan n  \to \deltan{n'}$ to the \oo-functor $f : \norient n \to \norient{n'}$ defined as follows:
\[
f \smp{i_0, \ldots, i_m} = \smp{\phi(i_0), \ldots, \phi(i_m)} \mbox{ for } 0 \leq i_0 < \cdots < i_m \leq n.
\]
\end{proposition}

\begin{proof}
By the description of the functor $\orient : \scat \to \ocat$ given in paragraph~\ref{paragr:free-monad},
it suffices to apply the formulas of paragraph~\ref{paragr:explicit-monad}.
\end{proof}

\section{Comparison with Street's orientals}
\label{comparison}

\subsection{Steiner's theory}

All the definitions and results presented in this subsection are extracted
from~\cite{Steiner}.

\begin{definition}[Augmented directed complexes]\ %

  \noindent
  An \ndef{augmented directed complex} $K$ consists of an augmented chain
  complex of abelian groups in non-negative degrees
  \[
  \xymatrix{
     \cdots \ar[r]^d & K_n \ar[r]^d & \cdots \ar[r]^d & K_1 \ar[r]^d & K_0
     \ar[r]^e & \Z
  }
  \]
  (meaning that we have $dd = 0$ and $ed = 0$) endowed with a submonoid
  $K_n^\ast$ of~$K_n$ for every $n \ge 0$.

  If $K$ and $L$ are two such augmented directed complexes, a
  \ndef{morphism} from $K$ to $L$ is a morphism $f$ of augmented chain
  complexes
  \[
  \xymatrix{
     \cdots \ar[r]^d & K_n \ar[r]^d \ar[d]_f & \cdots \ar[r]^d & K_1
     \ar[r]^d \ar[d]_f & K_0 \ar[d]_f \ar[r]^e & \Z \ar@{=}[d] \\
     \cdots \ar[r]^d & L_n \ar[r]^d & \cdots \ar[r]^d & L_1 \ar[r]^d & L_0
     \ar[r]^e & \Z \\
  }
  \]
  such that $f(K_n^\ast) \subset L^\ast_n$ for every $n \ge 0$.

  We will denote by $\ADC$ the category of augmented directed complexes.
\end{definition}

\begin{paragr}
  We define a functor $\lambda : \ooCat \to \ADC$ in the following way.
  Let $C$ be an \oo-category. For~$n \ge 0$, the abelian group
  $\lambda(C)_n$ is defined to be the quotient of the free abelian group on
  the set of $n$-cells of $C$ by the subgroup generated by the elements of
  the form $x \comp{j} y - x - y$, where $x, y$ is a pair of $n$-cells,
  $0 \le j < n$ and $x \comp{j} y$ is defined. We will denote by $[x]$ the
  image of an $n$-cell $x$ of $C$ in $\lambda(C)_n$. If $n > 0$, the
  differential $d : \lambda(C)_n \to \lambda(C)_{n-1}$ is defined on the
  generators by $d[x] = [\target{x}] - [\source{x}]$. If $n = 0$, the
  abelian group~$\lambda(C)_0$ is free on the set of objects of~$C$ and the
  augmentation $e : \lambda(C)_0 \to \Z$ is the sum of the coefficients.
  Finally, for $n \ge 0$, the monoid $K^\ast_n$ is the submonoid of $K_n$
  generated by the generators $[x]$.

  If $f : C \to D$ is an \oo-functor, then $\lambda(f)$ is defined on
  generators by $\lambda(f)[x] = [f(x)]$.

  One can check that these constructions are well-defined and indeed define
  a functor $\lambda$.
\end{paragr}

\begin{paragr}
  We now define a functor $\nu : \ADC \to \ooCat$.

  Let $K$ be an augmented directed complex. For $n \ge 0$, an $n$-cell of
  $\nu(K)$ is a table
  \[
      \begin{pmatrix}
        x_0^- & \cdots & x_n^- \\
        \noalign{\vskip 3pt}
        x_0^+ & \cdots & x_n^+ \\
      \end{pmatrix}
  \]
  where
  \begin{itemize}
    \item $x_i^\e$ belongs to $K_i^\ast$ for $\e = \pm$,
    \item $x^-_n = x^+_n$,
    \item $d(x_i^\e) = x^+_{i-1} - x^-_{i-1}$ for $0 < i \le n$ et $\e =
    \pm$,
    \item $e(x^\e_0) = 1$ for $\e = \pm$.
  \end{itemize}
  When $n > 0$, the source and target of such a table are given by the
  tables
  \[
      \begin{pmatrix}
        x_0^- & \cdots & x_{n-2}^- & x_{n-1}^- \\
        \noalign{\vskip 3pt}
        x_0^+ & \cdots & x_{n-2}^+ & x_{n-1}^- \\
      \end{pmatrix}
      \qquad\text{and}\qquad
      \begin{pmatrix}
        x_0^- & \cdots & x_{n-2}^- & x_{n-1}^+ \\
        \noalign{\vskip 3pt}
        x_0^+ & \cdots & x_{n-2}^+ & x_{n-1}^+ \\
      \end{pmatrix}
      .
  \]
  For $n \ge 0$, the identity of such a table is the table
  \[
      \begin{pmatrix}
        x_0^- & \cdots & x_{n}^- & 0 \\
        \noalign{\vskip 3pt}
        x_0^+ & \cdots & x_{n}^+ & 0 \\
      \end{pmatrix} .
  \]
  Finally, if
  \[
    x =
      \begin{pmatrix}
        x_0^- & \cdots & x_n^- \\
        \noalign{\vskip 3pt}
        x_0^+ & \cdots & x_n^+ \\
      \end{pmatrix}
    \qquad\text{and}\qquad
    y =
      \begin{pmatrix}
        y_0^- & \cdots & y_n^- \\
        \noalign{\vskip 3pt}
        y_0^+ & \cdots & y_n^+ \\
      \end{pmatrix}
  \]
  are two $n$-cells such that $\source_j{y} = \target_j{x}$ for a $j$ such
  that $0 \le j < n$, then the cell $y \comp{j} x$ is the table
  \[
      \begin{pmatrix}
        x_0^- & \cdots & x_j^- & x_j^- + y_j^- & \cdots & x_n^- + y_n^- \\
        \noalign{\vskip 3pt}
        y_0^+ & \cdots & y_j^+ & x_j^+ + y_j^+ & \cdots & x_n^+ + y_n^+ \\
      \end{pmatrix}.
  \]
  One can check that these cells and operations define an \oo-category $\nu(K)$.

  If $f : K \to L$ is a morphism of augmented directed complexes, then, for
  $n \ge 0$, the action of $\nu(f) : \nu(K) \to \nu(L)$ on $n$-cells is
  defined by
  \[
      \begin{pmatrix}
        x_0^- & \cdots & x_n^- \\
        \noalign{\vskip 3pt}
        x_0^+ & \cdots & x_n^+ \\
      \end{pmatrix}
      \mapsto
      \begin{pmatrix}
        f(x_0^-) & \cdots & f(x_n^-) \\
        \noalign{\vskip 3pt}
        f(x_0^+) & \cdots & f(x_n^+) \\
      \end{pmatrix}.
  \]

  One can check that these constructions define a functor $\nu : \ADC \to
  \ooCat$.
\end{paragr}

\begin{proposition}[Steiner]
  The functors
  \[
    \lambda : \ooCat \to \ADC
    \qquad
    \qquad
    \nu : \ADC \to \ooCat
  \]
  form a pair of adjoint functors.
\end{proposition}

\begin{proof}
  This is \cite[Theorem 2.11]{Steiner}.
\end{proof}

\begin{paragr}
  A \ndef{basis} of an augmented directed complex $K$ is a subgraded set
  $\coprod_{n \ge 0} B_n$ of $\coprod_{n \ge 0} K^\ast_n$ such that
  $B_n$ is a basis of the $\Z$-module $K_n$ that generates the monoid
  $K^\ast_n$. We will say that an augmented directed complex is \ndef{free}
  if it admits a basis.
\end{paragr}

\begin{remark}
  If such a basis exists, then we have, for $n \ge 0$,
  \[ K_n \simeq \Z^{(B_n)} \quad\text{and}\quad K^\ast_n \simeq \N^{(B_n)}. \]
  It follows from the second isomorphism that $B_n$ is uniquely determined
  by $K^\ast_n$. In other words, a free augmented directed complex admits a
  unique basis.
\end{remark}

\begin{paragraph}
  Let $K$ be a free augmented directed complex with basis $B$. Let $z$ be in
  $K_n$ for some $n \ge 0$. This element can be written in a unique way
  \[ z = \sum_{b \in B_n} z_b b , \]
  where $B_n$ is the basis of $K_n$. The \ndef{support} of $z$ is the subset
  of $B_n$ consisting of those $b$ such that~$z_b \neq 0$. We define
  $z^+$ and $z^-$ to be the unique elements of $K^\ast_n$ with disjoint
  supports such that
  \[ z = z^+ - z^-. \]

  For $n \ge 1$ and $x$ in $K_n$, we set
  \[ d^-(x) = d(x)^- \quad\text{and}\quad d^+(x) = d(x)^+, \]
  and, for $0 \le i < n$, we set
  \[ d^-_i(x) = (d^-)^{n-i}(x) \quad\text{and}\quad d^+_i(x) =
  (d^+)^{n-i}(x) . \]
  Note that $d^-_i(x)$ and $d^+_i(x)$ are elements of $K_i^\ast$.
\end{paragraph}

\begin{paragraph}
  Let $K$ be a free augmented directed complex with basis $B$. To any $x$ in
  $K^\ast_n$ for some~$n \ge 0$, we associate a table
  \[
      \begin{pmatrix}
        d_0^-(x) & \cdots & d_{n-1}^-(x) & x \\
        \noalign{\vskip 3pt}
        d_0^+(x) & \cdots & d_{n-1}^+(x) & x \\
      \end{pmatrix}.
  \]
  This table satisfies all the conditions to be an $n$-cell of $\nu(K)$
  except maybe that $e(d^-_0(x)) = 1$ and $e(d^+_0(x)) = 1$, where $e : K_0
  \to \Z$ is the augmentation of $K$.

  The complex $K$ is said to be \ndef{unital} if, for every $n \ge 0$ and
  every $x$ in $B_n$, we have $e(d^-_0(x)) = 1$ and $e(d^+_0(x)) = 1$. In
  this case, the table associated to $x$ is indeed an $n$-cell of $\nu(K)$
  that is called the \ndef{atom} of $x$.
\end{paragraph}

\begin{paragraph}
  A free augmented directed complex $K$ with basis $B$ is said to be
  \ndef{strongly loop-free} if there exists a partial order $\le$ on the set
  $B = \coprod_{n \ge 0} B_n$ such that, for every $x$ in $B_m$ and $y$
  in~$B_n$, if
  \[
    {
    \setstretch{0}
    \begin{cases}
      & \text{$m \ge 1$ and $y$ belongs to the support of $d^+(x)$} \\
      \text{or} \\
      & \text{$n \ge 1$ and $x$ belongs to the support of $d^-(y)$}
    \end{cases}
    }
  \]
  then we have $x \le y$.
\end{paragraph}

\begin{paragraph}
  A \ndef{strong Steiner complex} is a free augmented directed complex that is both unital and
  strongly loop-free.
\end{paragraph}

\begin{theorem}[Steiner]\label{thm:Steiner}
  The functor $\nu : \ADC \to \ooCat$ is fully faithful when restricted to
  strong Steiner complexes.
\end{theorem}

\begin{proof}
  This follows from \cite[Theorem 5.6 and Proposition 3.7]{Steiner}.
\end{proof}

\begin{theorem}[Steiner]\label{thm:Steiner_pol}
  For any strong Steiner complex $K$, the \oo-category $\nu(K)$ is freely
  generated, in the sense of polygraphs, by its atoms.
\end{theorem}

\begin{proof}
  This follows from \cite[Theorem 6.1 and Proposition 3.7]{Steiner}.
\end{proof}

\begin{remark}
  Steiner actually proved the two previous theorems for a more general
  class of complexes, where the strong loop-freeness condition is
  replaced
  by a weaker one (see \cite[Definition 3.5]{Steiner}).
\end{remark}

\subsection{A uniqueness result}

\begin{proposition}\label{prop:lambda_pol}
  If $S$ is a polygraph, then $\lambda(S^\ast)$ is free and its basis
  consists of the $[x]$, where $x$ varies among the generators of~$S$.
\end{proposition}

\begin{proof}
  Let $n \ge 0$. For every $x$ in $S_n$, we will denote by $e_x$ the
  corresponding element of the canonical basis of $\Z^{(S_n)}$.
  Consider the morphism $\gamma : \Z^{(S_n)} \to
  \lambda(S^\ast)_n$ defined by sending $e_x$ to~$[x]$, for every $x$ in
  $S_n$. We claim that this morphism is an isomorphism.
  Indeed, by \cite[paragraph 3.3]{MetResPol}, there exists a map $S^\ast_n
  \to \Z^{(S_n)}$ sending $x$ in $S_n$ to $[x]$ in $\Z^{(S_n)}$ and
  compositions in $S^\ast$ to sums. In particular, we get a morphism
  $\lambda(S^\ast)_n \to \Z^{(S_n)}$ sending $[x]$ in~$\lambda(S^\ast)_n$,
  for $x$ in~$S_n$, to $e_x$ in~$\Z^{(S_n)}$.  This morphism provides an
  inverse to $\gamma$.
\end{proof}

\begin{paragraph}
  Let $S$ be a polygraph.

  We say that a generator $x$ in $S_n$ is \ndef{atomic} if, for every $i$
  such that $0 \le i < n$, the supports of~$[\sce{x}_i]$ and $[\tge{x}_i]$
  are disjoint. The polygraph $S$ is said to be \ndef{atomic} if all its
  generators are atomic.

  The polygraph $S$ is \ndef{strongly loop-free} if there exists a partial
  order $\le$ on the generators of $S$ such that, for every $m \ge 0$ and $n
  \ge 0$, every $x$ in $S_m$ and $y$ in $S_n$, if
  \[
    {
    \setstretch{0}
    \begin{cases}
      & \text{$m \ge 1$ and $y$ belongs to the support of $[\tge x]$} \\
      \text{or} \\
      & \text{$n \ge 1$ and $x$ belongs to the support of $[\sce y]$}
    \end{cases}
    }
  \]
  then we have $x \le y$.

  Finally, $S$ is a \ndef{strong Steiner polygraph} if it is both atomic and
  strongly loop-free.
\end{paragraph}

Here is a reformulation of a result of Steiner based on \cite{AGOR}:

\begin{theorem}\label{thm:equiv_Steiner}
   The adjoint pair
   \[ \lambda : \ooCat \to \ADC \qquad \nu : \ADC \to \ooCat  \]
   induces an equivalence of categories between the full subcategory of
   $\ooCat$ consisting of \oo-categ\-ories freely generated by a strong Steiner
   polygraph and the full subcategory of $\ADC$ consisting of strong Steiner
   complexes.
 \end{theorem}

\begin{proof}
  This follows from \cite[Theorem 2.30]{AGOR}, based on \cite[Theorem
  5.11]{Steiner}.
\end{proof}

\begin{proposition}\label{prop:uniqueness}
  Let $S$ and $T$ be two polygraphs and let $f$ be a dimension-preserving
  bijection between the generators of~$S$ and the generators of $T$.
  Suppose that, for every $n \ge 1$ and every $x$ in $S_n$, we have
  \[
    f[\sce{x}] = [\sce{f(x)}]
    \quad\text{and}\quad
    f[\tge{x}] = [\tge{f(x)}].
  \]
  Suppose moreover that $S$ is a strong Steiner polygraph and that $T$ is
  atomic. Then $T$ is a strong Steiner polygraph and the map $f$ induces an
  isomorphism between $S^\ast$ and~$T^\ast$.
\end{proposition}

\begin{proof}
  If $x$ is in $S_n$, we have
  \[
     fd[x] = f([\tge{x}] - [\sce{x}]) = f[\tge{x}] - f[\sce{x}]
     = [\tge{f(x)}] - [\sce{f(x)}] = d\,\!f[x]
  \]
  and, by Proposition~\ref{prop:lambda_pol}, the map $f$ defines an
  isomorphism from $\lambda(S^\ast)$ to $\lambda(T^\ast)$. Using the
  previous theorem, to conclude the proof, it thus suffices to show that $T$
  is strongly loop-free. But being strongly loop-free only depends on the
  generators and on the operations $z \mapsto [\sce{z}]$ and $z \mapsto
  [\tge{z}]$, where $z$ is a generator. Since the bijection $f$ is compatible
  with these, we get the result.
\end{proof}

\subsection{Uniqueness of orientals}

We shall now give a ``linear characterization'' of $\norient{n}$, aiming at
proving it is isomorphic to Street's oriental. We saw in
paragraph~\ref{paragr:generators-On} that the $m$-generators of
$\norient{n}$ correspond to the injections $\deltan{m} \emb \deltan{n}$. We
now describe the linear source and target of such a generator:

\begin{proposition}\label{prop:lin_st_On}
  Fix $n \ge -1$. For every $m \ge 1$ and every $m$-generator $x$
  of~$\norient{n}$ considered as an injection $x : \deltan{m} \emb
  \deltan{n}$, we have
    \[
       [\sce{x}] = \sum_{\substack{0 \le i \le m\\\text{\rm $i$ odd}}} [x\faced{i}{m}]
       \quad\text{and}\quad
       [\tge{x}] = \sum_{\substack{0 \le i \le m\\\text{\rm $i$ even}}}
       [x\faced{i}{m}].
    \]
\end{proposition}

\begin{proof}
  We prove the result by induction on $n$. The assertion is clear if $n= -1$ or
  $n = 0$. Suppose $n > 0$. Let $m \ge 1$ and let $x = \smp{i_0, \dots,
  i_m}$ be a generator of $\norient{n}$ (see
  paragraph~\ref{paragr:generators-On}).
  \begin{enumerate}
    \item Suppose first that $i_0 \neq 0$. This means that $x = \eta(y)$,
    with $y =\smp{i_0-1, \dots, i_m-1}$ an $m$-generator of $\norient{n-1}$,
    where $\eta : \norient{n-1} \emb \norient{n}$ is the \oo-functor coming
    from the fact that $\norient{n}$ is the free expansion on
    $\norient{n-1}$. By induction, we have
    \[
       [\sce{y}] = \sum_{\substack{0 \le i \le m\\\text{$i$ odd}}} [y\faced{i}{m}]
       \quad\text{and}\quad
       [\tge{y}] = \sum_{\substack{0 \le i \le m\\\text{$i$ even}}}
       [y\faced{i}{m}],
    \]
    so that
    \[
      \begin{split}
        [\sce{x}]
        & = [\sce{\eta(y)}]
        = [\eta(\sce y)]
        = \lambda(\eta)[\sce y]
        \\
        & =
        \lambda(\eta)\Big(\sum_{\substack{0 \le i \le m\\\text{$i$ odd}}}
        [y\faced{i}{m}]\Big)
        =
        \sum_{\substack{0 \le i \le m\\\text{$i$ odd}}}
        [\eta(y\faced{i}{m})]
        \\
        &
        =
        \sum_{\substack{0 \le i \le m\\\text{$i$ odd}}}
        [\eta(y)\faced{i}{m}]
        =
        \sum_{\substack{0 \le i \le m\\\text{$i$ odd}}}
        [x\faced{i}{m}],
      \end{split}
    \]
    whence the desired formula, and similarly for $[\tge{x}]$.
    \item Suppose now that $i_0 = 0$. This means
    \[ x = \econe_{\eta(y)}, \]
    with $y =\smp{i_1-1, \dots, i_m-1}$ an $(m-1)$-generator of
    $\norient{n-1}$, where $\econe$ is the expansion of~$\norient{n}$.
    In particular,
    \[ \eta(y) = \smp{i_1, \dots, i_m} = x\faced{0}{m}. \]
    \begin{enumerate}
      \item If $m = 1$, so that $x = \smp{0, i}$, then $\eta(y) = \smp{i}$
      and
      \[ \econe_{\eta(y)} : \smp{0} \to \smp{i}. \]
      Thus
      \[
      [\sce{x}] = [\smp{0}] = [x\faced{1}{m}]
       \quad\text{and}\quad
      [\tge{x}] = [\smp{1}] = [x\faced{0}{m}],
      \]
      whence the result.
      \item If $m > 1$, then
      \[
      \econe_{\eta(y)} : \econe_{\tge{\eta(y)}_{m-1}} \!\! \to \eta(y) \comp 0
      \econe_{\sce{\eta(y)}_0} \comp 1 \cdots \comp{m-1} \econe_{\sce{\eta(y)}_{m-1}}, \]
      so that
      \[
        [\sce{x}] = [\econe_{\tge{\eta(y)}_{m-1}}]
        = [\econe_{\tge{\eta(y)}}]
        = [\econe_{\eta(\tge y)}]
        \]
        and
        \[
        [\tge{x}] = [\eta(y)] + [\econe_{\sce{\eta(y)}_{m-1}}]
        = [\eta(y)] + [\econe_{\eta(\sce{y})}],
      \]
      since $[z] = 0$ if $z$ is an identity.
      To be able to use this, we will need the fact that the oplax
      transformation $\econe$ induces a $\Z$-linear map (and actually even a
      chain homotopy)
      \[
        \begin{split}
        \lambda(\econe) :
          \lambda(\norient{n})_k & \to \lambda(\norient{n})_{k+1} \\
          [z] & \mapsto [\econe_z]
        \end{split}
      \]
      for every $k \ge 0$ (see the proof of Theorem~6.1 of
        \cite{MetResPol}). Now by induction, we have
      \[
       [\sce{y}] = \sum_{\substack{0 \le i \le m-1\\\text{$i$ odd}}}
         [y\faced{i}{m-1}]
       \quad\text{and}\quad
       [\tge{y}] = \sum_{\substack{0 \le i \le m-1\\\text{$i$ even}}}
         [y\faced{i}{m-1}],
      \]
      so that
      {
      \allowdisplaybreaks
      \begin{align*}
      [\sce{x}]
      & =
      [\econe_{\eta(\tge{y})}]
      =
      \lambda(\econe)\lambda(\eta)[\tge{y}]
      =
      \lambda(\econe)\lambda(\eta)
      \Big(
      \sum_{\substack{0 \le i \le m-1\\\text{$i$ even}}} [y\faced{i}{m-1}]
      \Big)
      \\
      &
      =
      \lambda(\econe)
      \Big(
      \sum_{\substack{0 \le i \le m-1\\\text{$i$ even}}}
      [\eta(y\faced{i}{m-1})]
      \Big)
      =
      \lambda(\econe)
      \Big(
      \sum_{\substack{0 \le i \le m-1\\\text{$i$ even}}}
      [\eta(y)\faced{i}{m-1}]
     \Big)
      \\
      &
      =
      \sum_{\substack{0 \le i \le m-1\\\text{$i$ even}}}
         [\econe_{\eta(y)\faced{i}{m-1}}]
      =
      \sum_{\substack{0 \le i \le m-1\\\text{$i$ even}}}
         [\econe_{\eta(y)}\faced{i+1}{m}]
      \\*
      & \phantom{= 1} \text{(the equality
      $\econe_{\eta(y)\faced{i}{m-1}} = \econe_{\eta(y)}\faced{i+1}{m}$
      being more transparent}
      \\*
      & \phantom{= 1 (} \text{under the form
      $\smp{0, \smp{i_1, \dots,
      i_m}\faced{i}{m-1}} = \smp{0, i_1, \dots, i_m}\faced{i+1}{m}$)} \\
      &
      =
      \sum_{\substack{0 \le i \le m-1\\\text{$i$ even}}}
         [x\faced{i+1}{m}]
=
      \sum_{\substack{1 \le j \le m\\\text{$j$ odd}}}
         [x\faced{j}{m}]
      =
      \sum_{\substack{0 \le j \le m\\\text{$j$ odd}}}
         [x\faced{j}{m}]
      \end{align*}
      }%
      as wanted. Similarly, one gets that
      \[
      [\econe_{\eta(\sce{y})}] =
      \sum_{\substack{1 \le j \le m\\\text{$j$ even}}}
         [x\faced{j}{m}]
      \]
      and hence that
      \[
      [\tge{x}] = [\eta(y)] + [\econe_{\eta(\sce{y})}]
      = [x\faced{0}{m}] + \sum_{\substack{1 \le j \le m\\\text{$j$ even}}}
      [x\faced{j}{m}]
      =
      \sum_{\substack{0 \le j \le m\\\text{$j$ even}}}
         [x\faced{j}{m}],
      \]
      thereby ending the proof.
      \qedhere
  \end{enumerate}
  \end{enumerate}
\end{proof}

\begin{proposition}
  For every $n \ge -1$, the polygraph defining $\norient{n}$ is atomic.
\end{proposition}

\begin{proof}
  We will prove more generally that if $S$ is an atomic polygraph, then so
  is the polygraph~$\expan S$ of paragraph~\ref{paragr:def_expan_S}.
  The result will then follow by induction as the polygraph defining
  $\norient{n}$ is obtained by iterating this construction from the empty
  polygraph, which is atomic.

  Let thus $S$ be an atomic polygraph. Consider a generator $x$
  of $\expan S$ of dimension $n \ge 1$.
  \begin{itemize}
    \item If $x = \eta(y)$ for $y$ a generator of $S$, where $\inc : S \to
    \expan S$ is the canonical morphism, then, as $y$ is atomic by
    hypothesis, so is $x$, as $\eta$ is injective on cells.

    \item Otherwise, $x = r_y$ for $y$ a generator of $S$, with the notation
    of paragraph~\ref{paragr:def_expan_S}. If $n = 1$, then
    \[ 
      \sce{(r_y)} = \orig
      \qquad\text{and}\qquad
      \tge{(r_y)} = y,
    \]
    where $\orig$ is the origin of $\expan S$. If $n > 1$, using the
    formulas
    \[
      \sce{(r_y)}= r_{\tge y_{n-2}}
      \qquad\text{and}\qquad
      \tge{(r_y)}= y\comp 0 r_{\sce y_0}\comp 1 \cdots
      \comp{n-2} r_{\sce y_{n-2}},
    \]
    where $y$ was identified with $\eta(y)$, we get by induction that, for
    $i$ such that $0 < i < n$,
    \[
      \sce{(r_y)}_i = r_{\tge y_{i-1}}
      \qquad\text{and}\qquad
      \tge{(r_y)}_i = \tge{y}_i \comp 0 r_{\sce y_0}\comp 1 \cdots
      \comp{i-1} r_{\sce y_{i-1}},
    \]
    and that
    \[
      \sce{(r_y)}_0 = \orig
      \qquad\text{and}\qquad
      \tge{(r_y)}_0 = \tge{y}_0.
    \]
    The supports of $\sce{(r_y)}_i$ and $\tge{(r_y)}_i$, for $0 \le i < n$,
    are thus disjoint and $x$ is atomic, whence the result.
    \qedhere
  \end{itemize}
\end{proof}

\begin{proposition}
  Fix $n \ge -1$ and let $S$ be an atomic polygraph such that
  \begin{enumerate}
    \item for every $m \ge 0$, we have $S_m = \{x : \deltan{m} \emb
    \deltan{n}\mid \text{$x$ injective and order-preserving}\}$,
    \item for every $m \ge 1$ and every $x : \deltan{m} \emb \deltan{n}$ in
    $S_m$, we have
    \[
       [\sce{x}] = \sum_{\substack{0 \le i \le m\\\text{\rm $i$ odd}}} [x\faced{i}{m}]
       \quad\text{and}\quad
       [\tge{x}] = \sum_{\substack{0 \le i \le m\\\text{\rm $i$ even}}}
       [x\faced{i}{m}].
    \]
  \end{enumerate}
  Then $S^\ast$ is canonically isomorphic to $\norient{n}$.
\end{proposition}

\begin{proof}
  By paragraph~\ref{paragr:generators-On} and
  Proposition~\ref{prop:lin_st_On}, these two properties are satisfied by
  the polygraph defining
  $\norient{n}$, which is atomic by the previous proposition.
  To get the result, using Proposition~\ref{prop:uniqueness}, 
  it thus suffices to produce a strong Steiner polygraph $S$ that satisfies
  these two properties. We could prove that the polygraph defining
  $\norient{n}$ does the job but it is simpler to refer to Steiner: the
  polygraph associated to the complex $\Delta[n]$ of \cite[Example
  3.8]{Steiner} satisfies these conditions.
\end{proof}

\begin{theorem}
  The cosimplicial object $\orient : \Delta \to \ooCat$ of
  Definition~\ref{def:orientals} is canonically isomorphic to the
  cosimplicial objects of orientals as introduced by Street in
  \cite{StreetOrient}.
\end{theorem}

\begin{proof}
  By \cite[Section 3 and Corollary 4.2]{StreetOrient}, the $n$-th oriental
  defined by Street satisfies the conditions of the previous proposition.
  This shows that the two cosimplicial objects agree on objects. To show
  that they also agree on morphisms, by using
  Theorem~\ref{thm:equiv_Steiner}, it suffices to show that they agree after
  applying $\lambda : \ooCat \to \ADC$. This follows from Proposition
  \ref{prop:cosimplicial} and~\cite[Section 5]{StreetOrient}.
\end{proof}

\bibliographystyle{alpha}
\bibliography{biblio}

\end{document}